\documentclass[11pt,reqno,oneside]{amsart}
\usepackage{amssymb,amsfonts,amsthm,amscd,stmaryrd,dsfont,esint,upgreek,constants,todonotes}
\usepackage{hyperref}
\usepackage{extsizes}
\usepackage[mathcal]{euscript}
\usepackage[latin1]{inputenc}   
\usepackage[mathcal]{euscript}
\usepackage{graphicx,color}
\usepackage{bm}
\newcommand\sbullet[1][.5]{\mathbin{\vcenter{\hbox{\scalebox{#1}{$\bullet$}}}}}
\numberwithin{equation}{section}

\newcommand{\R}{\mathbb R}

\def\P{\mathbb P}

\newconstantfamily{C}{symbol=C}
\newconstantfamily{c}{symbol=c}
\newconstantfamily{O}{symbol=\Omega}

\def\XXint#1#2#3{{\setbox0=\hbox{$#1{#2#3}{\int}$}
\vcenter{\hbox{$#2#3$}}\kern-.5\wd0}}

\numberwithin{equation}{section}
\newtheorem{thm}{Theorem}[section]
\newtheorem{lem}[thm]{Lemma}
\newtheorem{cor}[thm]{Corollary}
\newtheorem{prop}[thm]{Proposition}

\theoremstyle{definition}
\newtheorem{defn}[thm]{Definition}
\newtheorem{rmk}[thm]{Remark}

\def\smallnegint{\mathop{\int\mkern-13mu
        \raise.5ex\hbox{${\scriptscriptstyle\diagup}$}}\nolimits}

\def\ep{\varepsilon}

\def\F{{\mathcal F}}
\def\L{{\mathcal{L}}}

\def\tr{\operatorname{tr}}

\def\ssetminus{\,\raise.4ex\hbox{$\scriptstyle\setminus$}\,}

\newcommand{\be}{\begin{equation}}
\newcommand{\ee}{\end{equation}}

\newcommand{\bc}{\begin{case}}
\newcommand{\ec}{\end{cases}}
\newcommand{\bs}{\begin{split}}
\newcommand{\es}{\end{split}}

\newcommand{\Prob}{\mathcal{P}_2(\mathbb{R}^d)}
\newcommand{\Joint}{\mathcal{P}_2(\mathbb{R}^d\times\mathbb{R}^d)}

\renewcommand{\bar}{\overline}
\renewcommand{\tilde}{\widetilde}
\renewcommand{\hat}{\widehat}

\def\dw{{\bf d}_2}

\begin{document}
\title[Homogenization of conditional slow-fast McKean-Vlasov SDEs]{Homogenization of conditional slow-fast McKean-Vlasov SDEs}
\author[Antonios Zitridis]{Antonios Zitridis}
\address{Department of Mathematics, University of Chicago, Chicago, Illinois 60637, USA}
\email{zitridisa@uchicago.edu}
\vskip-0.5in 


\begin{abstract}
We study a fully-coupled system of conditional slow-fast McKean-Vlasov Stochastic Differential Equations that exhibit full dependence on both the slow and fast components, as well as on the conditional law of the slow component. Our aim is to derive convergence rates to its homogenized limit, without making periodicity assumptions. To prove our results, we utilize a perturbation method for equations posed in the Wasserstein space.


\vspace{1mm}


\end{abstract}

\maketitle
\vspace{-10mm}


\section{Introduction}
Let $(\Omega^1,\F^1,\mathbb{P}^1,\{\F^1_t\})$, $(\Omega^0,\F^0,\mathbb{P}^0,\{\F^0_t\})$ be filtered probability spaces with $\{\F^1_t\},\; \{ \F^0_t\}$, $B_t,\; W_t$ independent standard $m$-dimensional $\{\F^1_t\}$-Brownian motions and $B_t^0$ an $m$-dimensional $\{\F^0_t\}$ Brownian motion. For any parameter $\lambda\in \R$ small enough, we study the system of slow-fast McKean-Vlasov stochastic differential equations (SDEs for short)
\begin{align}\label{eq:slow-fastMcKeanVlasov}
dX^{\ep}_t &= \biggl(\frac{\lambda}{\ep}f(X^{\ep}_t,Y^{\ep}_t,\L(X_t^{\ep}|\F_t^0))+ c(X^{\ep}_t,Y^{\ep}_t,\L(X_t^{\ep}|\F_t^0)) \biggr)dt + \sigma(X^{\ep}_t,Y^{\ep}_t,\L(X_t^{\ep}|\F_t^0))dW_t\nonumber\\
 &\hspace{6cm}+\sigma^0_1(X_t^{\ep},\L(X_t^{\ep}|\F_t^0))dB_t^0\nonumber \\
dY^{\ep}_t & = \frac{1}{\ep}\biggl(\frac{1}{\ep}f(X^{\ep}_t,Y^{\ep}_t,\L(X_t^{\ep}|\F_t^0))+ g(X^{\ep}_t,Y^{\ep}_t,\L(X_t^{\ep}|\F_t^0)) \biggr)dt \\
&+\frac{1}{\ep}\biggl(\tau_1(X^{\ep}_t,Y^{\ep}_s,\L(X_t^{\ep}|\F_t^0))dW_t+\tau_2(X^{\ep}_t,Y^{\ep}_t,\L(X_t^{\ep}|\F_t^0))dB_t+ \sigma_2^0(X_t^{\ep},\L(X_t^{\ep}|\F_t^0))dB_t^0 \biggr),\nonumber
\end{align}
where $\L(X_t^{\ep}|\F_t^0)$ is the conditional law of $X_t^{\ep}$ and where $b,c,f,g:\R^d\times\R^d\times\Prob\rightarrow \R^d$, $\sigma,\tau_1,\tau_2:\R^d\times\R^d\times\Prob\rightarrow \R^{d\times m}$, $\sigma^0_1,\sigma_2^0:\R^d\times\Prob\rightarrow \R^{d\times m}$ $\eta\in L^2(\Omega^1,\F^1,\mathbb{P}^1;\R^d)$ with $X_0=\eta\sim \mu\in \Prob$, $Y_0=\zeta\in L^p(\Omega^1,\F^1,\mathbb{P}^1;\R^d)$ for all $p>0$, and $(\eta,\zeta)$ is independent of $(W,B)$. Note that this also means that $(\eta,\zeta)\sim m\in \Joint$ such that the $y$-marginal of $m$ has finite $p-$moments for every $p>0$ (we write $m\in\mathcal{A}$ for that). Here $\mathcal{P}_2(\R^n)$, $n\geq0$, denotes the space of probability measures over $\R^n$ with finite second moment, equipped with the 2-Wasserstein metric.\\

The theory of averaging for diffusion process has a long history broad applications in various fields, such as material sciences, chemistry, biology, and climate dynamics, among others (see e.g \cite{ran1, ran2, ran3, ran4}). Existing results of averaging for SDEs with coefficients that do not depend on the law of the solution are numerous two of which are by Pardoux and Veretenikov \cite{Pardoux,Pardoux2} and R\"ockner and Xie \cite{Rockner1}. So far in the literature, for slow-fast McKean-Vlasov SDEs (when the coefficients depend also on the law of the solution as in our case), the problem has been studied, with different methods, by Spiliopoulos, Bezemek \cite{Spil}, Hong, Wei, Sun \cite{Hong1}, Li, Wu, Xie \cite{Li}, Qiao, Wei \cite{Qiao} and R\"ockner, Sun, Xie  \cite{Rockner2}.\\
The aim of this paper is to describe the behavior of the slow component $X_t^{\ep}$ of \eqref{eq:slow-fastMcKeanVlasov} as $\ep\rightarrow 0$, in presence of the common noise $B_t^0$ this time, and establish that it converges to an averaged process, which we will find.\\

Let $a(x,y,\mu)=\frac{1}{2}[\tau_1\tau_1^\top+\tau_2\tau_2^\top+\sigma^0_2\sigma^{0\top}_2](x,y,\mu)$. For fixed $x\in\R^d$ and $\mu\in\Prob$, we consider the operator $L_0^{x,\mu}$ acting on $\phi\in C^2_b(\R^d)$ (that is $\phi$ has derivatives up to order two which are bounded) with
\be\label{L0 real}
L_0^{x,\mu}\phi =f(x,y,\mu)\cdot D\phi +a(x,y,\mu):D^2\phi(y).
\ee
Under our assumptions (see the corresponding section), by Proposition 1 from \cite{Pardoux}, there exists a $\pi(\cdot ; x,\mu)$ which is the unique probability measure solving the distributional equation 
\be\label{invariant}
\bigg(L_0^{x,\mu}\bigg)^*\pi =0.
\ee

Moreover, all moments of $\pi$ are bounded uniformly in $x,\mu$. Furthermore, there exists a classical solution to the equation
\be \label{Phi}
    L_0^{x,\mu}\Phi_k(x,y,\mu)=-\lambda f_k(x,y,\mu),\; k\in \{1,2,...,d\},\; x,y\in\R^d,\,\mu\in\Prob,
\ee
which is $\Phi_k(x,y,\mu)=-\lambda y_k,\; k\in \{1,...,d\}$. Now, let
\begin{align}\label{eq:limitingcoefficients}
\gamma(x,y,\mu)& :=c(x,y,\mu)-\lambda g(x,y,\mu)\\
D(x,y,\mu) & :=D_1(x,y,\mu)+D^\top_1(x,y,\mu)+\frac{1}{2}\sigma(x,y,\mu)\sigma^\top(x,y,\mu)-\frac{\lambda ^2}{2} \sigma_2^0\sigma_2^{0\top}(x,\mu) \nonumber\\
D_1(x,y,\mu)& := -\frac{\lambda}{2}[\lambda f(x,y,\mu)\otimes y+\tau_1(x,y,\mu)\sigma^\top(x,y,\mu)] \nonumber\\
\Sigma(x,\mu)& :=\sigma_1^0(x,\mu)-\lambda\sigma_2^0(x,\mu)\nonumber
\end{align}
and
\begin{align}\label{eq:averagedlimitingcoefficients}
\bar{\gamma}(x,\mu) := \biggl[\int_{\R^d} \gamma(x,y,\mu) \pi(dy;x,\mu)\biggr],\;\;
\bar{D}(x,\mu) :=\biggl[\int_{\R^d} D(x,y,\mu) \pi(dy;x,\mu)\biggr].
\end{align}

In this paper, we will study the convergence as well as the rate of convergence, of $\L(X^{\epsilon,m}_t|\F^0_t)$ to $\L(X^\mu_t|\F^0_t)$ on the space $\Prob$, where $X^\mu_t$ satisfies the averaged McKean-Vlasov SDE:
\begin{align}\label{eq:averagedMcKeanVlasov}
X^\mu_t = \eta^2+\int_0^t \bar{\gamma}(X^{\mu}_s,\L(X^{\mu}_s|\F^0_s))ds+&\int_0^t\sqrt{2}\bar{D}^{1/2}(X^{\mu}_s,\L(X^{\mu}_s|\F^0_s))dW^2_s\nonumber\\
&+\int_0^t\Sigma(X_s^{\mu},\L(X^{\mu}_s|\F^0_s))dB_s^0.
\end{align}

The equation is posed on a possibly different filtered probability space supporting a $d$-dimensional Brownian motion $W^2$, $X^{\epsilon,\mu}_0=\eta^2$ is random variable on this new probability space independent from $W^2$ with $\eta^2\sim \mu \in\Prob$. Finally, $\bar{D}^{1/2}(x,\mu)$ is the unique positive semi-definite matrix such that $\bar{D}^{1/2}(x,\mu)\bar{D}^{1/2}(x,\mu)=\bar{D}(x,\mu)$. Note that it is not clear that $\bar{D}$ is symmetric and positive semi-definite, therefore, as in \cite{Spil}, we will include it in our assumptions.\\

\subsection{Notation} $\sbullet[0.75]$ For any two matrices $A,B$ we denote $A:B=\tr(AB^T)$.\\
$\sbullet[0.75]$ For $i=1,2$, $\pi^i:\R^d\times\R^d\rightarrow \R^d$ is the projection $\pi^i(x_1,x_2)=x_i$.\\
$\sbullet[0.75]$ Let $X,\;Y$ and $Z \;(=\R\text{ or }\R^d)$ be metric spaces. For any function $h:X\times Y \times \Prob \rightarrow Z$ we set $\tilde{h}(x,y,m)=h (x,y,\pi^1_*m)$ for any $m\in \Joint$.\\
$\sbullet[0.75]$ If $X,Y,Z=\R^d$, we also write $\overline{h}(x,\mu)=\int h(x,y,\mu) \pi_{x,\mu}(dy)$, where $\pi_{x,\mu}$ is the probability measure obtained in \eqref{invariant}.\\
$\sbullet[0.75]$ We denote by $\mathcal{A}$ the set of all $m\in\Joint$ such that $\pi^2_*m$ has finite $p$-moments for every $p>0$. We note that $\mathcal{A}$ is convex.\\
$\sbullet[0.75]$ For $s\in [t,T]$, we will use the symbol $(X_s^{\ep},Y_s^{\ep})^{m,t}$ for the solution of \eqref{eq:slow-fastMcKeanVlasov} with initial joint distribution (for $s=t$) equal to $m$. Similarly, we define $X_s^{\mu,t}$ for the solution of \eqref{eq:averagedMcKeanVlasov}.\\
$\sbullet[0.75]$ We define the product structure $(\Omega,\F,\mathbb{P},\{ \F_t \})$, where $\Omega=\Omega^1\times \Omega^0$, $(\F,\mathbb{P})$ is the completion of $( \F^1\otimes \F^0, \mathbb{P}^1\otimes \mathbb{P}^0)$ and $\{\F_t\}$ is the complete right continuous augmentation of $\{ \F^1_t\otimes \F^0_t\}$.\\
$\sbullet[0.75]$ For $i=1,2$, we will denote by $\mathbb{E}^i$ the integration with respect to $\mathbb{P}^i$ and by $\mathbb{E}$ the integration with respect to $\mathbb{P}$.\\
$\sbullet[0.75]$ For the operator $L_0^{x,\mu}$ (or later $L_{0,2}^{x,x',\mu}$) we will sometimes abuse the notation and write $L_0^{x,\mu} \phi$ even if $\phi$ is vector or matrix valued meaning that we apply $L_0^{x,\mu}$ to every coordinate of $\phi$.

\subsection{The main result} The aforementioned convergence result involves tests functions $G$ which are at least 4 times continuously differentiable. The exact statement is the following.

\begin{thm} \label{main}
Assume Hypotheses (H1)-(H6) and let $T\in [0,+\infty)$, $m\in\mathcal{A}$ and $t\in [0,T]$. We consider a set of multi-indices $\bm{\dot{\zeta}}$ such that any $(n,\bm{\beta})$ with $|(n,\bm{\beta})|\leq 4$ is in $\bm{\dot{\zeta}}$.
Then, there exists a constant $C(T)$, independent of $\mu=\pi^1_*m$, such that for any $G\in \mathcal{M}_{b,L}^{\bm{\dot{\zeta}}}(\Prob;\R)$
$$\sup_{t\in[0,T]}| u^{\ep}(t,m)-u(t,\mu)|\leq \ep C(T)||G||_{\mathcal{M}_{b,L}^{\bm{\dot{\zeta}}}},$$
where $u$ is the unique classical solution of \eqref{eq:W2CauchyProblem} and $u^{\ep}$ is as in \eqref{uep}. Equivalently,
\be \label{final est}
\sup_{t\in [0,T]}\bigg|\mathbb{E}^0\bigg[ G(\L(X_T^{\ep,m,t}|\F^0_T))-G(\L(X_T^{\mu,t}|\F^0_T))\bigg ]\bigg |\leq \ep C(T)||G||_{\mathcal{M}_{b,L}^{\bm{\dot{\zeta}}}}.
\ee
\end{thm}

\vspace{4mm}

\textbf{The idea of the proof} is to show that the solution of the McKean-Vlasov control equation associated with \eqref{eq:slow-fastMcKeanVlasov} converges to the solution of the McKean-Vlasov control equation associated with the averaged process \eqref{eq:averagedMcKeanVlasov}. Indeed, as discussed in \cite{Pham} \footnote{In our case there is no control} the McKean-Vlasov control equation associated with the system \eqref{eq:slow-fastMcKeanVlasov} is 
\be \label{MFC eq}
\begin{cases}
\dot{u}^{\ep}(t,m) +\mathcal{L}^{\ep} u^{\ep}(t,m)=0,\;& \text{ in }[0,T)\times \Joint, \\
u^{\ep}(T,m)=G(\pi^1_*m),& \text{ in } \Joint,
\end{cases}
\ee
where 
\vspace{-2mm}

$$\mathcal{L}^{\ep}u(t,m)=\L_2u +\frac{1}{\ep}\L_1 u +\frac{1}{\ep^2}\L_0 u,\;\text{ with }$$

\begin{align}
\L_2u=& \int_{\R^{2d}} \bigg( \tilde{c}(x,y,m)\cdot D_x \frac{\delta u}{\delta m}++\frac{\tilde{\sigma\sigma^\top}(x,y,m)}{2}:D^2_{xx}\frac{\delta u}{\delta m}+\frac{\tilde{\sigma^0_1\sigma^{0\top}_1}(x,m)}{2}:D^2_{xx}\frac{\delta u}{\delta m}\bigg) m(dx,dy) \nonumber\\
& +\int_{\R^{2d}}\int_{\R^{2d}}\bigg( \frac{\tilde{\sigma_1^0}(x,m)\tilde{\sigma_1^{0\top}}(x',m)}{2}: D^2_{xx'}\frac{\delta^2 u}{\delta m^2}\bigg) m(dx,dy)m(dx',dy'), \label{L2}
\end{align}

\begin{align}
&\L_1u=\frac{1}{2}\int_{\R^{2d}}\int_{\R^{2d}}\bigg(\tilde{\sigma_2^0}(x,m)\tilde{\sigma_1^{0\top}}(x',m): D^2_{x'y}\frac{\delta^2 u}{\delta m^2}\bigg) m(dx,dy)m(dx',dy')\label{L1}\\
&\hspace{0.9cm}+\frac{1}{2}\int_{\R^{2d}}\int_{\R^{2d}}\bigg(\tilde{\sigma_1^0}(x,m)\tilde{\sigma_2^{0\top}}(x',m): D^2_{xy'}\frac{\delta^2 u}{\delta m^2}\bigg) m(dx,dy)m(dx',dy')\nonumber\\
&+\int_{\R^{2d}}\bigg( \tilde{f}(x,y,m)\cdot D_x\frac{\delta u}{\delta m}+\tilde{g}(x,y,m)\cdot D_y\frac{\delta u}{\delta m}+\tilde{\tau_1\sigma^\top}(x,y,m):D^2_{xy}\frac{\delta u}{\delta m}\nonumber\\
&\hspace{9.8cm}+\tilde{\sigma^0_1\sigma^{0\top}_2}(x,m):D^2_{yx}\frac{\delta u}{\delta m}\bigg) m(dx,dy) ,\nonumber
\end{align}

\begin{align}
&\L_0 u=\int_{\R^{2d}}\int_{\R^{2d}} \bigg( \frac{\tilde{\sigma_2^0}(x,m)\tilde{\sigma_2^{0\top}}(x',m)}{2}: D^2_{yy'}\frac{\delta^2 u}{\delta m^2}\bigg) m(dx,dy)m(dx',dy')\label{L0} \\
&+\int_{\R^{2d}}\bigg ( \tilde{f}(x,y,m)\cdot D_y\frac{\delta u}{\delta m}+\frac{(\tilde{\tau_1\tau_1^\top}+\tilde{\tau_2\tau_2^\top})(x,y,m)}{2}:D^2_{yy}\frac{\delta u}{\delta m}+\frac{\tilde{\sigma^0_2\sigma^{0\top}_2}(x,m)}{2}:D^2_{yy}\frac{\delta u}{\delta m} \bigg) m(dx,dy)\nonumber
\end{align}

and where the terminal condition $G$ is chosen to depend only on the $x$-marginal of $m$.\\  We know that, 
in general, 
\be \label{uep}
u^{\ep}(t,m)=\mathbb{E}^0[G(\mathcal{L}(X_T^{\ep,t,m}|\F^0_T))]
\ee
is a solution of \eqref{MFC eq} (see e.g \cite{Pham}, where $u^{\ep}$ is a viscosity solution). \\
The McKean-Vlasov control equation associated with the averaged dynamics \eqref{eq:averagedMcKeanVlasov} is
\be \label{eq:W2CauchyProblem}
\begin{cases}
\dot{u}(t,\mu)+ \L u(t,\mu)=0 ,& (t,\mu)\in [0,T)\times\Prob\\
u(T,\mu) = G(\mu),& \mu\in \Prob,
\end{cases}
\ee
with
\begin{align}
\L u(t,\mu)=\int_{\R^d}(\bar{\gamma}(x,\mu)\cdot D_\mu u(t,\mu,x) + (\bar{D}(x,\mu)+\frac{1}{2}\Sigma\Sigma^\top (x,\mu)):D_x D_\mu u(t,\mu,x))\mu(dx)\nonumber\\ 
\hspace{1cm} +\frac{1}{2}\int_{\R^{2d}}\Sigma(x,\mu)\Sigma^\top (x',\mu): D^2_{\mu\mu}u(t,\mu,x,x')\mu(dx)\mu(dx')\label{operator}
\end{align}
and where we choose $G$ to be the same terminal conditions as in \eqref{MFC eq}, and has solution $u(t,\mu)=\mathbb{E}^0[G(\L(X_T^{t,\mu}|\F^0_T))]$. Now, let $\mu$ be the $x$-marginal of $m$. We are expecting that, heuristically, under our assumptions on the coefficients, $u^{\ep}(t,m)$ will have an expansion of the form $u(t,\mu)+\ep u_1(t,m)+\ep^2 u_2(t,m)=u(t,\mu)+O(\ep)$, where $u_1,u_2$ will be some corrector functions. This, as $\ep\rightarrow 0$ will converge to $u(t,\mu)$, which gives us the desired convergence and the rate of convergence. The proof relies heavily on the regularity properties of the solutions of the McKean-Vlasov control equation \eqref{MFC eq} and the Poisson equation associated with \eqref{L0}.

\subsection{Organization of the paper}
In section 2, we state our assumptions on the coefficients as well as some preliminary properties related to \eqref{eq:slow-fastMcKeanVlasov}. Sections 3 and 4 contain some results on the McKean-Vlasov control equation and the Poisson problem, which will be useful in the proof. In section 5, we first construct the correctors $u_1,u_2$ and we use the results from sections 3 and 4 to to establish bounds on them and their derivatives. Finally, we use the properties of $u_1,u_2$ to finish the proof of Theorem \ref{main}. We use the appendix to present the technical proofs and some remarks on the main result of the paper.

\section{ Assumptions and Preliminary results}
Before we state the assumptions, we introduce the multi-index notation proposed in \cite{crisan} and used in \cite{Spil}.

\begin{defn} (Multi-index notation)\label{multi}
Let $n,l,k$ be natural numbers and $\bm{\beta}=(\beta_1,...,\beta_n)$ be an $n$-dimensional vector of natural numbers. Then, we call any ordered tuple $(n,l,\bm{\beta})$ or $(n,\bm{\beta})$ a multi-index. For a function $h:\R^d\times\Prob\rightarrow \R^k$, the $(n,l,\bm{\beta})$-derivative of $h$ is defined as
$$D^{(n,l,\bm{\beta})}h(x,m,z_1,...,z_n):= D_{z_n}^{\beta_n}...D_{z_1}^{\beta_1}D_x^lD_m^nh(x,m,z_1,...,z_n),$$
where $D_m$ is the intrinsic Wasserstein derivative (we refer to the Appendix). Similarly, for a function $h:\Prob\rightarrow \R^k$ we define
$$D^{(n,\bm{\beta})}h(m,z_1,...,z_n)=D_{z_n}^{\beta_n}...D_{z_1}^{\beta_1}D_m^nh(m,z_1,...,z_n).$$
\end{defn}
\vspace{2mm}
In the spaces we will define below, we will be using the symbol $\bm{\zeta}$ for a collection of multi-indices. We will say that $\bm{\zeta}$ is a \textbf{complete} collection of multi-indices if for every $(n,l,\bm{\beta})\in \bm{\zeta}$ we have that $\bm{\zeta}$ contains all lower-order multi-indices of the same type. For example, if $(2,0,(1,1))\in\bm{\zeta}$, then $(2,0,(1,0)), (2,0,(0,1)),(1,0,1),(1,0,0)$ and $(0,0,0)$ are also in $\bm{\zeta}$.
\vspace{2mm}

\begin{defn}
Let $\bm{\zeta}$ be a collection of multi-indices of the form $(n,l,\bm{\beta})$ and $k,j,d\in\mathbb{N}$ ($k$ may also be in $\mathbb{N}\times\mathbb{N}$). We define $\mathcal{M}_b^{\bm{\zeta}}(\R^j\times\Prob;\R^k)$ to be the class of functions $h:\R^j\times\Prob\rightarrow \R^k$ such that the $(n,l,\bm{\beta})-$derivative of $h$ exists and satisfies
$$||h||_{\mathcal{M}_b^{\bm{\zeta}}}:=\sup_{(n,l,\bm{\beta})\in\bm{\zeta}}\sup_{x\in\R^j,z_1,...,z_n\in\R^d,m\in\Prob}|D^{(n,l,\bm{\beta})}h(x,m,z_1,...,z_n)|\leq C.$$
We also write that $h\in\mathcal{M}_{b,L}^{\bm{\zeta}}$, if for all $(n,l,\bm{\beta})\in \bm{\zeta}$, $x,x'\in\R^j$, $z_1,...,z_n,z_1',...,z_n'\in\R^d$ and $m,m'\in\Prob$
$$|D^{(n,l,\bm{\beta})}h(x,m,z_1,...,z_n)-D^{(n,l,\bm{\beta})}h(x',m',z_1',...,z_n')|\leq C_L\bigg( |x-x'|+\sum_{i=1}^n|z_i-z_i'|+\dw(m,m')\bigg).$$
We can define the spaces $\mathcal{M}_b^{\bm{\zeta}}(\Prob;\R^k)$ and $\mathcal{M}_{b,L}^{\bm{\zeta}}(\Prob;\R^k)$ analogously, where $\bm{\zeta}$ will now be a collection of multi-indices of the form $(n,\bm{\beta})$, as explained in the previous definition.\\
Furthermore, we define the class $\mathcal{M}_p^{\bm{\zeta}}(\R^j\times\R^j\times \Prob ;\R^k)$ which consists of functions $h:\R^j\times\R^j\times\Prob\rightarrow \R^k$ such that $h(\cdot,y,\cdot)\in \mathcal{M}_b^{\bm{\zeta}}$ for all $y\in\R^j$ and for each multi-index $(n,l,\bm{\beta})\in \bm{\zeta}$
there exists $m\in\mathbb{N}$ and $C>0$ such that 
$$\sup_{x\in \R^j, z_1,...,z_n\in \R^d,m\in \Prob}|D^{(n,l,\bm{\beta})}h(x,y,m,z_1,...,z_n)|\leq C(1+|y|^m)$$
and there exists $\theta\in (0,1]$ (independent of the multi-index) such that
\begin{multline*}
|D^{(n,l,\bm{\beta})}h(x,y,m,z_1,...,z_n)-D^{(n,l,\bm{\beta})}h(x',y',m',z_1',...,z_n')|\leq\\
C\bigg( 1\wedge |y-y'|^{\theta}+|x-x'|+\sum_{i=1}^n|z_i-z_i'|+\dw (m,m')\bigg)\bigg(1+|y|^m+|y'|^m\bigg).
\end{multline*}
Moreover, we introduce $\mathcal{M}_b^{\bm{\zeta}}([0,T]\times\R^j\times\Prob;\R^k)$ which is the class of functions $h:[0,T]\times \R^j\times\Prob\rightarrow \R^k$ such that $h(\cdot,x,m)$ is continuously differentiable on $(0,T)$ for all $x\in\R^j,m\in\Prob$ with time derivative denoted by $\dot{h}(t,x,m)$, $h(t,\cdot,\cdot)\in \mathcal{M}_b^{\bm{\zeta}}$ for all $t\in [0,T]$ with $||h(t,\cdot,\cdot)||_{\mathcal{M}_b^{\bm{\zeta}}}\leq C$ uniformly in $t$ and all the derivatives of $h$ and $\dot{h}$ are jointly continuous with respect to the time, space and measure variables. The norm in this space is defined as:
$$||h||_{\mathcal{M}_b^{\bm{\zeta}}([0,T]\times \R^j\times\Prob;\R^k)}:=\sup_{t\in [0,T]}||h(t,\cdot)||_{\mathcal{M}_b^{\bm{\zeta}}}+\sup_{t\in [0,T], x\in \R^j,m \in\Prob} |\dot{h}(t,x,m)|.$$
Finally, we can define the spaces $\mathcal{M}_{b,L}^{\bm{\zeta}}([0,T]\times\R^j\times\Prob;\R^k)$, $\mathcal{M}_b^{\bm{\zeta}}([0,T]\times\Prob;\R^k)$, $\mathcal{M}_b^{\bm{\zeta}}([0,T]\times\Prob; \R^k)$ and $\mathcal{M}_p^{\bm{\zeta}}([0,T]\times\R^j\times\R^j\times\Prob;\R^k)$ similarly.
\end{defn}

\begin{rmk}
(i) Notice that for the class $\mathcal{M}_p^{\bm{\zeta}}(\R^j\times\R^j\times \Prob ;\R^k)$ and $\mathcal{M}_p^{\bm{\zeta}}([0,T]\times \R^j\times\R^j\times \Prob ;\R^k)$ , we may choose $m,C$ to be independent of the multi-index, because the set of multi-indices is finite. So for the function $h$ we will write $m=p_h$ and $C=C_h$. For both classes, we thus have $||h||_{\mathcal{M}_p^{\bm{\zeta}}}\leq C_h(1+|y|^{p_h})$.\\
(ii) In the above definitions we may have $\mathcal{A}$ or any convex subset of $\mathcal{P}_2(\R^d)$ instead of the whole $\mathcal{P}_2(\R^d)$.
\end{rmk}

\subsection{Assumptions}

Let $\hat{\bm{\zeta}}$ be a set of multi-indices such that every $(n,l,\bm{\beta})$ such that $|(n,l,\bm{\beta})|=n+l+\beta_1+...+\beta_n\leq 4$ is in $\hat{\bm{\zeta}}$. The following hypotheses hold:

(H1) There exist $\alpha,\beta>0$ such that 
$$0<\alpha|z|^2\leq z^Ta(x,y,\mu)z\leq \beta|z|^2<\infty,\;\text{for all }x,y,z\in \R^d,\; z\neq 0,\; \mu\in\Prob$$
and $\tau_1,\tau_2$ are bounded, have two uniformly bounded derivatives in $y$ and both these derivatives are H\"older continuous in $y$ uniformly in $(x,\mu)$.\\

(H2) There exist constant $C_1,C_2>0$ independent of $x,y,\mu$ such that 
$$f(x,y,\mu)\cdot y\leq -C_1|y|^2+C_2,\;\text{for all }x,y\in\R^d,\;\mu\in\Prob,$$
$f$ grows at most linearly in $|y|$, $f$ has two uniformly bounded derivatives in $y$ and both these derivatives are H\"older continuous in $y$ uniformly in $(x,\mu)$.\\

(H3) $\sigma^0_1,\sigma^0_2$ are in $\mathcal{M}^{\hat{\bm{\zeta}}}_b(\R^d\times\Prob;\R^{d\times d})$.\\

(H4) The functions $\tau_2,\tau_1, \sigma, c,f,g,\sigma^0_1,\sigma^0_2$ are globally Lipschitz continuous in $(x,y,\mu)$. In addition $g$ is uniformly bounded.\\

(H5) $g,c\in \mathcal{M}_p^{\bm{\hat{\zeta}}}(\R^d\times\R^d\times\Prob;\R^d)$, $\sigma,\tau_1\in \mathcal{M}_p^{\bm{\hat{\zeta}}}(\R^d\times\R^d\times\Prob;\R^{d\times m})$,\\
$a\in \mathcal{M}_p^{\bm{\hat{\zeta}}}(\R^d\times\R^d\times\Prob;\R^{d\times d})$ and $f\in \mathcal{M}_p^{\bm{\hat{\zeta}}}(\R^d\times\R^d\times\Prob;\R^d)$.\\

(H6) For $\lambda=0$, $\bar{D}$ in equation \eqref{eq:averagedlimitingcoefficients} is positive, that is there exists an $\alpha '>0$ such that $z^T\bar{D}(x,\mu)z\geq \alpha' |z|^2>0$ for every $x,z\in \R^d,\; z\neq 0,\; \mu\in \Prob$.\\

\begin{rmk} \label{remark for appendix}
(i) We note that Assumption (H4) is sufficient for the existence and the uniqueness of a solution of \eqref{eq:slow-fastMcKeanVlasov} (see e.g \cite{Rockner1}).\\
(ii) Assumptions (H1)-(H3) are useful in the solution of the Poisson equation. Specifically, we use them in Proposition \ref{poisson for calculations}. Assumption (H2) is also used in the proof of Proposition \ref{inv est}.\\
(iii) The amount of regularity we have on the coefficients (assumptions (H3), (H4), (H5)) is needed to establish the analogous regularity of the averaged coefficients ($\bar{\gamma}, \bar{D}^{1/2}, \Sigma$) and, in consequence, of the solution $u$ of the averaged McKean-Vlasov control equation \eqref{eq:W2CauchyProblem}.\\
(iv) The last assumption (H6) is stated because it is not clear whether $\bar{D}$ is a positive matrix. Due to (H3), (H6) also implies that $\bar{D}$ is again positive if $\lambda$ is small enough. Note that the positivity of $\bar{D}$ implies that $\bar{D}^{1/2}$ exists and is positive. Given that Proposition \ref{poisson for calculations} implies that $\overline{D}\in \mathcal{M}_{b,L}^{\hat{\bm{\zeta}}}$, the positivity of $\overline{D}$ yields $\overline{D}^{1/2}\in \mathcal{M}_{b,L}^{\hat{\bm{\zeta}}}$.\\
(v) The boundness of $g,\;\tau_1,\;\tau_2$ and $\sigma^0_2$ (mentioned in (H1), (H3) and (H4)) is used in the proof of Proposition \ref{inv est}.\\
(vi) In some previous works on the problem without common noise, e.g \cite{Spil}, the coefficient of $\frac{1}{\ep}$ in the slow variable in \eqref{eq:slow-fastMcKeanVlasov} is chosen to be a function $b$, instead of a multiple of $f$, which appears in the fast variable as well. This choice will be explained in the appendix.
\end{rmk}

\vspace{4mm}

\subsection{Preliminary results}
A preliminary observation about system \eqref{eq:slow-fastMcKeanVlasov} is that the distribution of $(X^{\ep}_t,Y_t^{\ep})$ in $\Omega^1\times\Omega^0$, $t>0$, remains in $\mathcal{A}$ provided that the distribution of $(X_0^{\ep},Y_0^{\ep})$ is in $\mathcal{A}$. We hence have the following proposition.

\begin{prop}\label{inv est}
Suppose that (H1)-(H4) hold and $T>0$. Let $(X_t^{\ep},Y_t^{\ep})$ be the solution of \eqref{eq:slow-fastMcKeanVlasov} with initial condition $(X_0,Y_0)=(\eta,\zeta)\sim m\in \mathcal{A}$; that is $\eta\in L^2(\Omega^1,\F^1,\P;\R^d)$  and $\zeta\in L^p(\Omega^1,\F^1,\P;\R^d)$ for every $p>0$. Then, for any $p\in \mathbb{N}$ the following estimate holds
\be\label{est}
\sup_{\ep\in (0,1]}\sup_{t\in [0,T]}\mathbb{E}\bigg[ |Y_t^{\ep}|^{2p} \bigg]\leq C(p)+\mathbb{E}^1[ |\zeta|^{2p}],
\ee
for some constant $C(p)$ depending on $p$. In particular, the conditional expectation $\L( X_t^{\ep},Y_t^{\ep}|\F^0_t)$ is in $\mathcal{A}$, $\mathbb{P}^0-$a.s for any $\ep\in (0,1]$ and any $t\in  [0,T]$.
\end{prop}

\begin{proof}
Fix $p\in\mathbb{N}$. We first note that since $\zeta\in L^{2p}(\Omega^1)$ we may construct $(X_t^{\ep},Y_t^{\ep})$ such that $\mathbb{E}(|Y_t^{\ep}|^{2p})$ is finite for all $t\in [0,T]$. We only have to apply It\^o's formula to $|Y^{\ep}_t|^{2p}$ and we take the expectation $\mathbb{E}$ to derive
\begin{align*}
\mathbb{E}\bigg[|Y_t^{\ep}|^{2p}\bigg]&=\mathbb{E}[|\zeta|^{2p}]+ \frac{2p}{\ep}\int_0^t\mathbb{E}\bigg[g\cdot Y_s^{\ep}|Y_s^{\ep}|^{2p-2}   \bigg ]ds
+\frac{2p}{\ep^2}\mathbb{E}\bigg[ \text{tr}(a)|Y_s^{\ep}|^{2p-2} \bigg]ds\\
&+\frac{2p}{\ep^2}\int_0^t\mathbb{E}\bigg[ f\cdot Y_s^{\ep}|Y_s^{\ep}|^{2p-2}+2(p-1)|Y_s^{\ep}|^{2p-2} a:Y_s^{\ep}\otimes Y_s^{\ep}   \bigg]ds,
\end{align*}
where $a$ is the matrix defined at the introduction before relation \eqref{L0 real}. Since $g,\; a$ are bounded (by Hypotheses (H1), (H3) and (H4)) and $f$ satisfies (H2), the proof now follows as in Lemma 4.1 of \cite{Spil}. 
\end{proof}

\section{The Poisson Equation}
To prove Theorem \ref{main} we will use some regularity results for the Poisson equation associated with the operator \eqref{L0}. In the following results $\bm{\zeta}$ is a complete collection of multi-indices that contains all $(n,l,\bm{\beta})$ with $|(n,l,\bm{\beta})|\leq 2$.

\vspace{2mm}

\begin{prop}\label{poisson for calculations}
Assume (H1) and (H2) and suppose that $h\in \mathcal{M}_p^{\bm{\zeta}}(\R^d\times\R^d\times \Prob;\R)$, $f\in \mathcal{M}_p^{\bm{\zeta}}(\R^d\times\R^d\times \Prob;\R^d)$ and $a\in \mathcal{M}_p^{\bm{\zeta}}(\R^d\times\R^d\times \Prob;\R^{d\times d})$. Then, 
$\bar{h}(x,\mu):=\int_{\R^d}h(x,y,\mu)\pi_{x,\mu}(dy)\in \mathcal{M}_{b,L}^{\bm{\zeta}}(\R^d\times \Prob;\R)$.\\
Moreover, there exists a unique classical solution $H$ to the problem
$$
\begin{cases}
L_0^{x,\mu}H(x,y,\mu)=h(x,y,\mu)-\bar{h}(x,\mu),\\
\int_{\R^d}H(x,y,\mu)\pi_{x,\mu}(dy)=0,
\end{cases}
$$
such that $H$ has at most polynomial growth as $|y|\rightarrow \infty$ and the mixed derivatives
$H, D_{y_i}H$ and $D_{y_iy_j}^2H$ are in $\mathcal{M}_p^{\bm{\zeta}}(\R^d\times\R^d\times \Prob;\R)$ for every $i,j\in \{1,...,d\}$.
\end{prop}

\begin{proof}
The first part of the statement follows from Lemma A.6 of \cite{Spil}. The second part is Lemma A.5 from the same paper. Note that, regarding the derivatives with respect to the measure variable, even though the reference is using the Lions derivative, by regularity and due to Remark \ref{Wass rem} (iii), we may use the notion of the Wasserstein derivative (Definition \ref{Wass}).
\end{proof}
\vspace{2mm}

With the idea of doubling the variables in \eqref{L0}, we may consider the operator $L_{0,2}^{x,x',\mu}$ acting on $\phi\in C^2_b(\R^d\times\R^d)$ with
\begin{align}
L_{0,2}^{x,x',\mu}\phi(y,y'):=&f(x,y,\mu)\cdot D_y\phi(y,y')+f(x',y',\mu)\cdot D_{y'}\phi(y,y')\nonumber \\
&+a(x,y,\mu):D^2_{yy}\phi(y,y')+a(x',y',\mu):D^2_{y'y'}\phi(y,y')\nonumber \\
&+\sigma^0_2(x,\mu)\sigma^{0\top}_2(x',\mu):D^2_{yy'}\phi(y,y'). \label{L02}
\end{align}
We observe that due to the positivity of $a$, $L_{0,2}^{x,x',\mu}$ is non-degenerate elliptic. We also notice that the unique distributional solution $\Pi_{x,x',\mu}$ to the adjoint equation 
$$
\begin{cases}
\bigg(L_{0,2}^{x,x',\mu}\bigg)^*\Pi_{x,x',\mu}=0,\\ 
\int_{\R^{2d}}\Pi_{x,x',\mu}(dy,dy')=1,\;\text{ for every }x,\;x'\in\R^d\text{ and }\mu\in\Prob
\end{cases}
$$
has $y$-marginal equal to $\pi_{x,\mu}$ and $y'$-marginal equal to $\pi_{x',\mu}$. Indeed, this follows by observing that $L_{0,2}^{x,x',\mu}\phi(y)=L_0^{x,\mu}\phi(y)$ and $L_{0,2}^{x,x',\mu}\phi(y)=L_0^{x',\mu}\phi(y')$, where $L_0^{x,\mu}$ was introduced in \eqref{L0 real}.

\begin{lem} \label{double poisson}
Assume Hypotheses (H1)-(H6). Let $\varphi : [0,T]\times \R^{4d}\times\Prob \rightarrow \R$ be an element of $\mathcal{M}_p^{\bm{\zeta}}([0,T]\times \R^{2d}\times\R^{2d} \times \Prob ;\R)$. We also assume that 
\be \label{condition}
\int_{\R^{2d}} \varphi(t,x,x',y,y',\mu)\Pi_{x,x',\mu}(dy,dy')=0,\;\text{for every }t\in[0,T],\; x,x'\in\R^d\text{ and }\mu\in \Prob.
\ee
Then, there exists a function $v:[0,T]\times\mathcal{A} \rightarrow \R $ such that for every $(t,m)\in [0,T]\times \mathcal{A}$
\be \label{eq for lemma 1}
\L_0 v(t,m)= \int_{\R^{2d}}\int_{\R^{2d}}\varphi(t,x,x,y,y',\pi^1_*m)m(dx,dy)m(dx',dy').
\ee
\end{lem}

\begin{proof}
By Proposition \ref{poisson for calculations} applied to the operator $L_{0,2}^{x,x',\mu}$, there exists a function $\chi\in \mathcal{M}_p^{\bm{\zeta}}([0,T]\times \R^{2d}\times\R^{2d} \times \Prob ;\R)$ such that 
\be \label{aux poisson}
L_0^{x,x',\mu}\chi(t,x,x',y,y',\mu)=\varphi(t,x,x',y,y',\mu).
\ee
We claim that $v:[0,T]\times\mathcal{A}\rightarrow \R$ with 
\be \label{definition of v}
v(t,m)= \int_{\R^{2d}}\int_{\R^{2d}}\chi(t,x,x',y,y',\pi^1_*m)m(dx,dy)m(dx',dy')
\ee
satisfies \eqref{eq for lemma 1}. Indeed, since $\chi\in \mathcal{M}_p^{\bm{\zeta}}$ we have 
\begin{align}
\frac{\delta v}{\delta m}(t,m,x,y&)= \int_{\R^{2d}} \chi(t, x,x',y,y',\pi^1_*m)m(dx',dy')\nonumber\\
&+\int_{\R^{2d}}\chi(t,x',x,y',y,\pi^1_*m)m(dx',dy')\nonumber\\
&+ \int_{\R^{4d}}\frac{\delta \chi}{\delta m}(t,\bar{x},x',\bar{y},y',\pi^1_*m, x) m(d\bar{x},d\bar{y})m(dx',dy').\label{first der of v}
\end{align}
In addition, 
\begin{align}
\frac{\delta^2 v}{\delta m^2}(t,m,x,y,x',y')&= \chi(t,x,x',y,y',\pi^1_*m)+\chi(t,x',x,y',y,\pi^1_*m)\nonumber\\
&+ \int_{\R^{2d}}\frac{\delta \chi}{\delta m} (t,x,\bar{x},y,\bar{y},\pi^1_*m,x')m(d\bar{x},d\bar{y})\nonumber\\
&+\int_{\R^{2d}}\frac{\delta \chi}{\delta m}(t,\bar{x},x,\bar{y},y,\pi^1_*m,x')m(d\bar{x},d\bar{y})\nonumber\\
&+ \int_{\R^{4d}}\frac{\delta^2 \chi}{\delta m^2}(t,\bar{x},\tilde{x},\bar{y},\tilde{y},\pi^1_*m, x,x') m(d\bar{x},d\bar{y})m(d\tilde{x},d\tilde{y})\nonumber \\
&+ \int_{\R^{4d}}\frac{\delta \chi}{\delta m}(t,\bar{x},x',\bar{y},y',\pi^1_*m, x)m(d\bar{x},d\bar{y})\nonumber\\
&+\int_{\R^{4d}}\frac{\delta \chi}{\delta m}(t,x',\bar{x},y',\bar{y},\pi^1_*m, x)m(d\bar{x},d\bar{y}). \label{second der of v}
\end{align}
Notice that only the first two terms depend on both $y,\;y'$. Therefore,
\begin{align}
\L_0 v(t,m&)=\int_{\R^{2d}}\int_{\R^{2d}} L_0^{x,\mu}\chi(t, x,x',y,y',\pi^1_*m) m(dx',dy')m(dx,dy)\nonumber\\
& +\int_{\R^{2d}}\int_{\R^{2d}} L_0^{x,\mu}\chi(t, x',x,y',y,\pi^1_*m) m(dx',dy')m(dx,dy)\nonumber\\
&+\int_{\R^{4d}}\frac{ \sigma^0(x,\pi^1_*m)\sigma^{0\top}(x',\pi^1_*m)}{2}: D^2_{yy'}\chi(t,x,x',y,y',\pi^1_*m)m(dx,dy)m(dx',dy')\nonumber\\
&+\int_{\R^{4d}}\frac{ \sigma^0(x,\pi^1_*m)\sigma^{0\top}(x',\pi^1_*m)}{2}: D^2_{yy'}\chi(t,x',x,y',y,\pi^1_*m) m(dx,dy)m(dx',dy').\label{go back}
\end{align}
We observe that, due to $D^2_{yy'}\chi(t,x',x,y',y,\pi^1_*m)$ being a symmetric matrix, the fourth term is equal to 
\begin{multline*}
\int_{\R^{4d}}\frac{ \sigma^0(x',\pi^1_*m)\sigma^{0\top}(x,\pi^1_*m)}{2}: D^2_{yy'}\chi(t,x',x,y',y,\pi^1_*m)\; m(dx,dy)m(dx',dy')=\\ \int_{\R^{4d}}\frac{ \sigma^0(x,\pi^1_*m)\sigma^{0\top}(x',\pi^1_*)}{2}: D^2_{yy'}\chi(t,x,x',y,y',\pi^1_*m)\;m(dx',dy')m(dx,dy),
\end{multline*}
which is equal to the third term. Also, after renaming the variables, the second term is equal to $\int_{\R^{2d}}\int_{\R^{2d}} L_0^{x',\mu}\chi(t, x,x',y,y',\pi^1_*m) m(dx',dy')m(dx,dy)$. Thus, going back to \eqref{go back},
$$\L_0 v(t,m)= \int_{\R^{4d}} L_{0,2}^{x,x',\mu}\chi(t,x,x',y,y',\mu)\;m(dx,dy)m(dx',dy').$$
Combining this with \eqref{aux poisson} yields \eqref{eq for lemma 1}.
\end{proof}

Finally, we prove the following lemma which will provide us with useful estimates when we introduce the correctors $u_1,\; u_2$.
\vspace{2mm}

\begin{lem} \label{estimates}
Assume that Hypotheses (H1)-(H6) hold and let $\chi : [0,T]\times \R^{4d}\times\Prob \rightarrow \R$ be an element of $\mathcal{M}_p^{\bm{\zeta}}([0,T]\times \R^{2d}\times\R^{2d} \times \Prob ;\R)$. We define the function $v:[0,T]\times \mathcal{A}\rightarrow \R$ as in \eqref{definition of v}. Then for every $m\in\mathcal{A}$ there exists a constant $C$ depending only on $T$, $\pi^2_*m$ and $\chi$ such that 
\be \label{main estimate for correctors}
\sup_{t\in [0,T]}\bigg\{ |v(t,m)|+|\dot{v}(t,m)|+|\L_1 v(t,m)|+|\L_2 v(t,m)| \bigg\}\leq C,
\ee
where $\L_1,\; \L_2$ are as in \eqref{L1} and \eqref{L2}.
\end{lem}

\begin{proof}
We, essentially, have to prove four estimates. For any $(t,m)\in [0,T]\times\mathcal{A}$ we have, by the definition of the space $\mathcal{M}^{\bm{\zeta}}_p$ and Remark 2.3 (i)
\begin{align}
&|v(t,m)|\leq  C\int_{\R^{4d}}(1+|y|^{p_{\chi}}+|y'|^{p_{\chi}})m(dx,dy)m(dx',dy')\leq C\bigg(1+\int_{\R^d} |y|^{p_{\chi}}\pi^2_*m(dy)\bigg),\label{3.10}\\
&\text{ and  likewise}\nonumber\\
& |\dot{v}(t,m)|\leq \int_{\R^{4d}} |\dot{\chi}(t,x,x',y,y',\pi^1_*m)|m(dx,dy)m(dx',dy')\leq  C\bigg(1+\int_{\R^d} |y|^{p_{\chi}}\pi^2_*m(dy)\bigg).\label{3.11}
\end{align}
To bound $|\L_1v|$ and $|\L_2v|$ we will use \eqref{first der of v}, \eqref{second der of v}. It is easy to see that after plugging \eqref{first der of v}, \eqref{second der of v} in $\L_1v$ and $\L_2 v$ as given by \eqref{L1} and \eqref{L2} we get two sums of terms of the form $\int_{\R^{4d}}\phi(t,x,x',y,y',\pi^1_*)m(dx,dy)m(dx',dy')$ with $\phi$ being a function depending linearly on the derivatives of $\chi$ and the coefficients of \eqref{eq:slow-fastMcKeanVlasov}. Since all these functions are in some $\mathcal{M}_p$, we deduce that any such $\phi$ is in some $\mathcal{M}_p^{\bm{\zeta}'}$ (here $\bm{\zeta'}$ is a set of multi-indices) and hence we can bound any term
$$\bigg|\int_{\R^{4d}}\phi(t,x,x',y,y',\pi^1_*)m(dx,dy)m(dx',dy')\bigg|\leq C\bigg(1+\int_{\R^d} |y|^{p}\pi^2_*m(dy)\bigg)$$
for some $p$ depending on $\chi$ and the coefficients of \eqref{eq:slow-fastMcKeanVlasov}. Thus, for $(t,m)\in [0,T]\times\mathcal{A}$ we deduce the bound
\begin{align}
|\L_1v(t,m)|+|\L_2v(t,m)|\leq C\bigg(1+\int_{\R^d} |y|^{p}\pi^2_*m(dy)\bigg)\label{3.12}
\end{align}
for some $p\in \mathbb{N}$. Since $m\in\mathcal{A}$ (therefore it has finite moments), we combine \eqref{3.10}, \eqref{3.11} and \eqref{3.12} and \eqref{main estimate for correctors} follows.
\end{proof}

\begin{rmk} \label{Remark 1}
Due to the linear dependence of all the terms of $v,\; \dot{v},\; \L_1v$ and $\L_2v$ on $\chi$, it follows from the proof that the constant in \eqref{main estimate for correctors} is a multiple of $C_{\chi}$, where $C_{\chi}$ is defined in Remark 2.3.
\end{rmk}

\section{Cauchy problem associated to McKean-Vlasov processes}

As mentioned in the introduction we consider the McKean-Vlasov control equation associated with \eqref{eq:averagedMcKeanVlasov} which is \eqref{eq:W2CauchyProblem}. The following theorem holds.

\begin{thm} \label{McKean-Vlasov thm}
Assume Hypotheses (H1)-(H6) and that $G\in \mathcal{M}_{b,L}^{\bm{\dot{\zeta}}}(\Prob;\R)$, where $\bm {\dot{\zeta}}$ is as in Theorem \ref{main}. Then, there exists a unique solution $u\in \mathcal{M}_{b,L}^{\bm{\dot{\zeta}}}([0,T]\times \Prob;\R)$ of \eqref{eq:W2CauchyProblem} with 
\be \label{u est}
||u||_{\mathcal{M}_b^{\bm{\dot{\zeta}}}}\leq C(T)||G||_{\mathcal{M}_b^{\bm{\dot{\zeta}}}},
\ee
for some constant $C(T)$ that depends only on $T$, the norms of the coefficients of \eqref{eq:W2CauchyProblem}: $||\bar{\gamma}||_{\mathcal{M}_b^{\bm{\hat{\zeta}}}} $, $||\bar{D}^{1/2}||_{\mathcal{M}_b^{\bm{\hat{\zeta}}}}$ and their Lipschitz constants.\\
Furthermore  $u(t,\mu)=\mathbb{E}^0[G(\L(X_T^{\mu,t}|\F^0_T)]$ and the following estimate holds
\be \label{u est t-der}
\sup_{\substack{t\in [0,T],\mu\in \Prob \\x_1,x_2\in \R^d}}\bigg \{ |D_m\dot{u}(t,\mu,x_1)|+|D_{x_1}D_m\dot{u}(t,\mu, x_1)|+|D^2_m\dot{u}(t,\mu,x_1,x_2)|\bigg\} \leq C(T)||G||_{\mathcal{M}_b^{\bm{\dot{\zeta}}}}.
\ee
\end{thm}
\begin{proof}
First we observe that $\overline{\gamma}, \overline{D}\in \mathcal{M}_p^{\hat{\bm{\zeta}}}$ follows from Proposition \ref{poisson for calculations}. In addition, since, by (H6), $\overline{D}$ is positive and the square root map is Fr\'echet differentiable up to arbitrary order with all derivatives being bounded on sets of uniformly positive matrices, we also have that $\overline{D}^{1/2}\in \mathcal{M}_p^{\hat{\bm{\zeta}}}$ (see also \cite{Chen} for a more detailed discussion). Now, the result follows from Theorem \ref{main appendix}.
\end{proof}

\section{Proof of the main result}
In this section we will prove the main result of this paper. For motivation purposes, we will first show a heuristic calculation which sketches the proof of the convergence described in section 1.2. We will be using the notation $\mu=\pi^1_*m$.

\subsection{Heuristic calculation}
We start by making the ansatz $u^{\ep}(t,m)=u(t,m)+\ep u_1(t,m)+\ep^2 u_2(t,m)$ in \eqref{MFC eq} and we collect the $\mathcal{O}(\frac{1}{\epsilon^2}),\mathcal{O}(\frac{1}{\epsilon})$ and $\mathcal{O}(1)$ terms. Consequently, we write
\begin{align*}
\frac{\L_0 u}{\ep^2}+\frac{ \L_1u+\L_0 u_1}{\ep}+ \mathcal{L}_0u_2+\mathcal{L}_1u_1+\mathcal{L}_2u-\L u
+\dot{u} +\L u=\mathcal{O}(\ep).
\end{align*}
We want the following relations to hold
\begin{align}
\mathcal{O}\bigg(\frac{1}{\epsilon^2}\bigg) &:\;\; \mathcal{L}_0u=0,\label{O(ep2)}\\
\mathcal{O}\bigg(\frac{1}{\epsilon}\bigg)&:\;\; \mathcal{L}_0u_1+\mathcal{L}_1u=0,\label{O(ep)}\\
\mathcal{O}(1)&: \;\; \mathcal{L}_0u_2=-\mathcal{L}_1u_1-\mathcal{L}_2u+\L u.\label{O(1)}
\end{align}
The first equation is satisfied when $\frac{\delta u}{\delta m}(t,m,x,y)$ is independent of $y$, so by Proposition \ref{Wass independence}, $u(t,m)=u(t,\pi^1_*m)$; that is, $u(\cdot,m)$ depends only on $\pi^1_*m$. The second equation, can be written as
$$\mathcal{L}_0u_1(t,m)=-\lambda\int_{\R^{2d}}\tilde{f}(x,y,m)\cdot D_m u(t,\pi^1_*m,x)m(dx,dy),$$
so we expect $u_1(t,m)=-\lambda\int_{\R^{2d}}y\cdot D_mu(t,\mu,x)\; m(dx,dy)$. Now, due to Lemma \ref{double poisson} there exists a solution $u_2$ to the third equation. Putting everything together, we deduce that $u$ satisfies
$$\dot{u}(t,\mu)+\L u (t,\mu)=\mathcal{O}(\ep),$$
thus, letting $\ep\rightarrow 0$, we deduce that $u^{\ep}$ converges to $u$, which satisfies the equation in \eqref{eq:W2CauchyProblem}.

\subsection{Proof of the main theorem}
We will construct the corrector $u_1$ first.

\begin{prop} \label{first corr}
Assume Hypotheses (H1)-(H6) and let $u$ be the classical solution of \eqref{eq:W2CauchyProblem} given by Theorem \ref{McKean-Vlasov thm}. Then, there exists a function $u_1: [0,T]\times \mathcal{A} \rightarrow \R$ such that 
\be \label{first corr eq}
\L_0 u_1(t,m)=-\L_1 u(t,m)
\ee
in the classical sense. Moreover, for every $m\in\mathcal{A}$ there exists a constant $C>0$ depending only on $T$ and $\pi_*^2m$ such that
\begin{align}
\sup_{t\in [0,T]}\bigg\{ |u_1(t,m)|+|\dot{u_1}(t,m)|+|\L_2 u_1(t,m)| \bigg\}\leq C ||G||_{\mathcal{M}_{b,L}^{\bm{\dot{\zeta}}}}.\label{first corr est}
\end{align}
\end{prop}

\begin{proof}
First, we note that since $u(t,m)$ depends only on the $x-$marginal of $m$, by Proposition \ref{Wass reg} $\frac{\delta u}{\delta m}(t,m,x)$ is independent of $y$, therefore $\L_1u= \lambda\int_{\R^{2d}} f(x,y,\mu)\cdot D_{\mu}u(t,\mu,x)m(dx,dy)$. For $(t,m)\in [0,T]\times \Joint$ we define $u_1$ to be such that
$$u_1(t,m)= -\lambda\int_{\R^{2d}}y\cdot D_m u(t,\pi_*^1m,x)\; m(dx,dy).$$
By \cite{CarmonaDelarue} (section 5.2.2 Example 3) and Proposition \ref{Wass reg}, since the dependence of the integrand with respect to the measure is on $\pi^1_*m$, we have
\be\nonumber
\frac{\delta u_1}{\delta m}(t,m,x,y)= -\lambda y\cdot D_m u(t,\pi_*^1m,x)-\lambda \int_{\R^{2d}} 
y'\cdot \frac{\delta (D_m u) }{\delta m}(t,x',\pi^1_*m,x)\; m(dx',dy'),
\ee
while $\frac{\delta^2 u_1}{\delta m^2}(t,m,x,y,x',y')$ has no term depending on both $y$ and $y'$. So, using \eqref{L0}, we get $\L_0 u_1=-\lambda\int_{\R^{2d}}f(x,y,\pi^1_*m)\cdot D_m u(t,\pi^1_*m,x)m(dx,dy)=-\L_1u(t,m)$, which is \eqref{first corr eq}.\\
To prove \eqref{first corr est} we will apply Lemma \ref{estimates} for $\chi_1(t,x,x',y,y',\mu):=y\cdot D_mu(t,\mu,x)$; that is $\chi_1$ is constant in $x',\; y'$. We observe, by Theorem \ref{McKean-Vlasov thm} that $\chi_1\in \mathcal{M}_p^{\bm{\zeta}_1}([0,T]\times \R^d\times\R^d\times\Prob;\R)$, where $\bm{\zeta_1}$ is a set of multi-indices containing all $(n,l,\bm{\beta})$ with $|(n,l,\bm{\beta})|\leq 3$, as as indicated by Theorem \ref{McKean-Vlasov thm}. In addition, by \eqref{u est}, $C_{\chi_1}= C(T)||G||_{\mathcal{M}_p^{\bm{\dot{\zeta}}}}$. Having this in mind, estimate \eqref{first corr eq} follows from Lemma \ref{estimates} and Remark \ref{Remark 1}.
\end{proof}


Next, we will construct the second corrector $u_2$.

\begin{prop} \label{second corr}
Suppose that hypotheses (H1)-(H6) hold and let $u, u_1$ be as in Proposition \ref{first corr}. Then, there exists a function $u_2: [0,T]\times\mathcal{A} \rightarrow \R$ such that 
\be \label{second corr eq}
\L_0 u_2(t,m)= -\L_1u_1-\L_2u +\L u.
\ee
Furthermore, for every $m\in\mathcal{A}$ there exists a constant $C>0$ depending only on $T$ and $\pi^2_*m$ such that 
\be \label{second corr est}
\sup_{\substack{t\in [0,T], \\ m\in\mathcal{A}}}\bigg\{ |u_2(t,m)|+|\dot{u}_2(t,m)|+ |\L_1 u_2(t,m)|+|\L_2 u_2(t,m)| \bigg\}\leq C||G||_{\mathcal{M}_{b,L}^{\bm{\dot{\zeta}}}}.
\ee
\end{prop}

\begin{proof}
After plugging in \eqref{L1}, \eqref{L2}, \eqref{operator} and the formula for $u_1$ in the right hand side of \eqref{second corr eq} a straightforward simplification yields
\begin{align}
&\;\;\;\;\int_{\R^{2d}}(\bar{c}(x,\mu)-c(x,y,\mu))\cdot D_mu (t,\mu,x) m(dx,dy)\label{5.8}\\
&+\lambda\int_{\R^{2d}}(g(x,y,\mu)-\bar{g}(x,\mu))\cdot D_{\mu}u(t,m,x) m(dx,dy) \\
&+\int_{\R^{2d}}\bigg(\frac{\bar{\sigma\sigma^\top}(x,\mu)-\sigma\sigma^\top(x,y,\mu)}{2}\bigg):D_xD_{\mu}u(t,\mu,x)m(dx,dy)\\
&+\lambda^2\int_{\R^{2d}} \bigg ( \frac{y\otimes f(x,y,\mu)+(y\otimes f(x,y,\mu))^\top-\bar{y\otimes f(x,y,\mu)+(y\otimes f(x,y,\mu))^\top}}{2} \bigg )\\
&\hspace{10cm}:D_xD_{\mu}u(t,\mu,x) m(dx,dy)\nonumber\\
&+\lambda\int_{\R^{2d}}\bigg( \frac{\tau_1\sigma^\top+\sigma\tau_1^\top}{2}(x,y,\mu)-\bar{\frac{\tau_1\sigma^\top+\sigma\tau_1^\top}{2}}(x,\mu)\bigg) :D_xD_{\mu}u(t,\mu,x)m(dx,dy)\\
&+\lambda ^2\int_{\R^{2d}}\int_{\R^{2d}}f(x,y,\mu)\cdot D^2_{\mu\mu}u(t,\mu,x,x')y'\;m(dx',dy')m(dx,dy)\label{5.14}\\
&+\frac{\lambda ^2}{2}\int_{\R^{2d}}\int_{\R^{2d}}\sigma^0_2(x,\mu)\sigma^{0\top}_2(x',\mu):D^2_{\mu\mu}u(t,
\mu,x,x')\mu(dx)\mu(dx').\label{5.15}
\end{align}

We note that all the integrands satisfy \eqref{condition} except for the last two. Hence, by Proposition \ref{poisson for calculations}, there exist $u_{i,2}\in\mathcal{M}_p^{\bm{\hat{\zeta}}}(\R^d\times\R^d\times \Prob; \R^d)$, $i\in \{1,2\}$ and $u_{i,2}\in\mathcal{M}_p^{\bm{\hat{\zeta}}}(\R^d\times\R^d\times \Prob; \R^{d\times d})$, $i=3,4,5$ such that
\begin{align}
&L_{0,2}^{x,x',\mu}u_{2,1}=\bar{c}(x,\mu)-c(x,y,\mu)\\
&L_{0,2}^{x,x',\mu}u_{2,2}=\lambda (g(x,y,\mu)-\bar{g}(x,\mu) )\\
&L_{0,2}^{x,x',\mu}u_{2,3}=\lambda^2\frac{y\otimes f(x,y,\mu)+(y\otimes f(x,y,\mu))^\top-\bar{y\otimes f(x,y,\mu)+(y\otimes f(x,y,\mu))^\top}}{2}\\
&L_{0,2}^{x,x',\mu}u_{2,4}=\frac{\bar{\sigma\sigma^\top}(x,\mu)-\sigma\sigma^\top(x,y,\mu)}{2}\\
&L_{0,2}^{x,x',\mu}u_{2,5}=\lambda\frac{\tau_1\sigma^\top+\sigma\tau_1^\top}{2}(x,y,\mu)-\lambda\bar{\frac{\tau_1\sigma^\top+\sigma\tau_1^\top}{2}}(x,\mu).
\end{align}
We also, have for $u_{2,6}(t,\mu,x,x',y,y')=\frac{\lambda ^2}{2}y\cdot D^2_{\mu\mu}u(t,\mu,x,x')y'$
\begin{align*}
L_{0,2}^{x,x',\mu} u_{2,6}= &\frac{\lambda ^2}{2}f(x,y,\mu)\cdot D^2_{\mu\mu}u(t,\mu,x,x')y'+\frac{\lambda ^2}{2}y\cdot D^2_{\mu\mu}u(t,\mu,x,x')f(x',y',\mu)\\
&+\frac{\lambda ^2}{2}\sigma_2^0(x,\mu)\sigma_2^{0\top}(x',\mu):D^2_{\mu\mu}u(t,\mu,x,x').
\end{align*}
Note that from Theorem \ref{McKean-Vlasov thm}, it follows that $u_{2,6}\in \mathcal{M}_p^{\bm{\zeta}}([0,T]\R^{2d}\times\R^{2d}\times\Prob;\R)$ with $\bm{\zeta}$ as described in section 3. In addition, since $D^2_{\mu\mu}u$ is symmetric with respect to $x,x'$ (Lemma 2.4 from \cite{Cardaliaguet master eq}), we deduce that $\int_{\R^{4d}}L_{0,2}^{x,x',\mu} u_{2,6}$ is equal to the sum of the quantities \eqref{5.14} and \eqref{5.15}. For $(t,m)\in [0,T]\times\mathcal{A}$ we now define
\begin{align} 
u_2(t,m&):=\int_{\R^{4d}} u_{2,6}(t,\mu,x,x',y,y')m(dx,dy)m(dx',dy')\nonumber\\
&+ \sum_{i=1}^2 \int_{\R^{2d}} u_{2,i}(\mu,x,y)\cdot D_{\mu}u(t,\mu,x)m(dx,dy)\nonumber\\
&+\sum_{i=3}^5 \int_{\R^{2d}} u_{2,i}(\mu,x,y):D_xD_{\mu}u(t,\mu,x)m(dx,dy).\label{second corrector}
\end{align}
By repeating the calculation from Lemma \ref{double poisson}, due to the construction of $u_{2,i}$, $i\in \{1,2,...,6\}$ we deduce that
$\L_0 u_2$ is equal to the sum of the quantites \eqref{5.8}-\eqref{5.15}. Therefore, we derive \eqref{second corr eq}.\\
The estimate \eqref{second corr est}, follows from Lemma \ref{estimates}. Indeed, due to the regularity of the $u_{2,i},\; i=1,2,...,6$ and Theorem \ref{McKean-Vlasov thm}, Lemma \ref{estimates}
can be applied to every term of $u_2$. The quantity $C||G||_{\mathcal{M}_{b,L}^{\bm{\dot{\zeta}}}}$ appears on the right hand side of \eqref{second corr est} as in the proof of Proposition \ref{first corr} because of Remark \ref{Remark 1}.
\end{proof}

We are now ready to prove Theorem \ref{main}

\begin{proof}
We will formalize the heuristic argument provided in the previous subsection. Let $u$ be the unique classical solution of \eqref{eq:W2CauchyProblem}. We consider
\be \label{vep}
v^{\ep}(t,m)= u(t,\pi^1_*m)+\ep u_1(t,m)+\ep^2 u_2(t,m),
\ee
where $u_1,u_2$ are given by Propositions \ref{first corr} and \ref{second corr}. We will exploit the fact that $v^{\ep}$
is ``almost'' a solution to \eqref{MFC eq}, as the heuristic calculation in section 5.1 indicates.\\

Let $m\in \mathcal{A}$ and $t\in [0,T]$. For $s\in [t,T]$, we denote by $m_s$ the conditional joint law $\mathcal{L}((X_s^{\ep},Y_s^{\ep})^{t,m}|\F^0_s)$ of the pair $(X_t^{\ep},Y_t^{\ep})$ satisfying \eqref{eq:slow-fastMcKeanVlasov} given that the initial distribution (at $s=t$) is $m$, and we also set $\mu_s=\pi^1_*m_s$. Note that since $m\in \mathcal{A}$, by Proposition \ref{inv est}, we deduce that $m_s\in\mathcal{A}$, $\mathbb{P}^0$-a.s for any $s\in [t,T]$ and $\ep\in (0,1)$. We observe that by the construction of $v^{\ep}$ and It\^o's formula (Theorem 4.14 of \cite{carmdel 2}) we have
\begin{align*}
\mathbb{E}^0(v^{\ep}(T,m_T))&= \mathbb{E}^0( v^{\ep}(t,m))+\mathbb{E}^0\bigg[\int_t^T\bigg(\dot{v}^{\ep}(s,m_s)+\L^{\ep} v^{\ep}(s,m_s)\bigg)ds\bigg]\\
&= \mathbb{E}^0( v^{\ep}(t,m))+ \mathbb{E}^0\bigg[\int_t^T\bigg(\dot{u}(s,\mu_s)+\L u(s,\mu_s)\bigg)ds\bigg]\\
&+\mathbb{E}^0\bigg[ \int_t^T( \ep \dot{u}_1(s,m_s)+\ep^2 \dot{u}_2(s,m_s) )ds   \bigg]\\
&+\mathbb{E}^0\bigg[\int_t^T\bigg( \ep^2 \L_2 u_2(s,m_s)+\ep \L_1 u_2(s,m_s) +\ep\L_2 u_1(s,m_s)\bigg) ds\bigg].
\end{align*}
Hence, since $u$ satisfies \eqref{eq:W2CauchyProblem},
\begin{align}
\mathbb{E}^0[v^{\ep}(T,m_T)-v^{\ep}(t,m)]=\mathbb{E}^0\bigg[\int_t^T\bigg( \ep \dot{u}_1+\ep^2 \dot{u}_2+\ep^2 \L_2 u_2+\ep \L_1 u_2 +\ep\L_2 u_1\bigg)(s,m_s) ds\bigg].
\end{align}
Now, by Proposition \ref{inv est}, $m_s\in\mathcal{A}$, so using the estimates \eqref{first corr est} and \eqref{second corr eq}, we deduce
\be \label{useful}
\bigg|\mathbb{E}^0\bigg[ v^{\ep}(T,m_T)-v^{\ep}(t,m) \bigg]\bigg|\leq \ep C(T)||G||_{\mathcal{M}_b^{\bm{\dot{\zeta}}}},
\ee
where $C(T)$ is some constant depending on $T$ and $\pi^2_*m$. Finally, it follows from the definition of $u,\; v^{\ep}$ and the triangle inequality that 
\begin{align*}
\bigg|&\mathbb{E}^0\bigg[G(\mu_T)-G(\L(X_T^{t,\mu}|\F_T^0))\bigg]\bigg|= \bigg|\mathbb{E}^0\bigg[u(T,\mu_T)-u(t,\mu)\bigg]\bigg| \\
&\leq \bigg|\mathbb{E}^0\bigg[v^{\ep}(T,m_T)-v^{\ep}(t,m)\bigg]\bigg|+\bigg|\mathbb{E}^0\bigg[\ep u_1(T,m_T)-\ep u_1(t,m) +\ep^2 u_2(T,m_T)-\ep^2 u_2(t,m)\bigg]\bigg|\\
&\leq \ep C(T)||G||_{\mathcal{M}_b^{\bm{\zeta}}},
\end{align*}
where in the last inequality we used \eqref{useful}, \eqref{first corr est} and \eqref{second corr est}. Thus, we derive \eqref{final est}. The proof is complete.   
\end{proof}

\appendix\section {The Wasserstein space and differentiation}

This section is devoted to the Wasserstein space and the notion of differentiation defined on it. We will focus on the Wasserstein space $\mathcal{P}_2$, even though the results can be extended for $\mathcal{P}_p$, $p>0$. Let $m_1,m_2\in\mathcal{P}_2(\R^n)$ the metric we will be using is given by 
$$\dw (m_1,m_2)=\inf_{\substack{ \gamma\in \mathcal{P}_2(\R^n\times\R^n)\\ \pi^1_*\gamma=m_1,\; \pi^2_*\gamma=m_2}}\int_{\R^{2n}}|x-y|^2\gamma(dx,dy).$$

\begin{defn} \label{Wass}
Let $U:\mathcal{P}_2(\R^n)\rightarrow \R$, $n\geq 1$. \\
(i) We say that $U$ is a $C^1$ function if there exists a continuous map $K:\mathcal{P}_2(\R^n)\times \R^n\rightarrow \R$ such that for all $m_1,m_2\in \mathcal{P}_2(\R^n)$ we have $\int_{\R^n}K(m_1,x)m_2(dx)<\infty$ and
\be \nonumber
\lim_{t\rightarrow 0} \frac{ U((1-t)m_1+tm_2)-U(m_1)}{t}=\int_{\R^n}K(m_1,x)(m_2-m_1)(dx).
\ee
We will use the symbol $\frac{\delta U}{\delta m}(m,x)$ for the map $K(m,x)$, which is called the Wasserstein flat derivative.\\
(ii) If $U$ is of class $C^1$ and $\frac{\delta U}{\delta m}$ is $C^1$ with respect to the space variable, we define the intrinsic derivative (Wasserstein derivative) $D_mU: \mathcal{P}_2(\R^n)\times \R^n\rightarrow \R^n$ as
$D_mU(m,x):=D_x\frac{\delta U}{\delta m}(m,x).$
\end{defn}
\vspace{1mm}

\begin{rmk} \label{Wass rem}
(i) Note that in the first part of the above definition, the Wasserstein flat derivative is defined up to a constant. \\
(ii) We can adopt the same definition for a function $U$ defined on a convex subset of $\mathcal{P}_2(\R^n)$ with values in $\R$. In particular, $U$ may be defined on the set $\mathcal{A}$ we introduced in section 1.1.\\
(iii) In our assumptions (section 2), due to the regularity, our intrinsic derivative coincides with the Lions derivative (introduced by Lions in his lectures at Coll\`ege de France \cite{Lions}) used in \cite{Rockner1, Li, Hong1,crisan}. Indeed, since the functions $h(x,y,\mu)$ we will be working with have $C^1$ and bounded (with respect to $x'$) Wasserstein flat derivative $\frac{\delta h}{\delta m}(x,y,\mu,x')$ and the intrinsic derivative $D_mh(x,y,\mu,x')$ is jointly continuous, by Proposition 5.48 of \cite{CarmonaDelarue}, they have a Lions derivative which coincides with the Wasserstein instrinsic derivative. Conversely, any function $h(\mu)$ with a continuous and uniformly (in $\mu$) Lipschitz Lions derivative $\partial_{\mu}h(\mu)(x')$, by Proposition 5.51 of \cite{CarmonaDelarue} has a Wasserstein flat derivative $\frac{\delta h}{\delta m}(\mu,x')$.
\end{rmk}

\vspace{2mm}

We continue with some properties of the functions defined over $\mathcal{P}_2(\R^n)$, $n\geq 1$, in relation to composition with projections.
Let $\pi^1,\pi^2: \R^d\times\R^d\rightarrow \R^d$ be the projections $\pi^1(x_1,x_2)=x$ and $\pi^2(x_1,x_2)=x_2$. It known that if $m\in\Joint$, then $\pi^1_*m$ and $\pi^2_*m$ are the $x$ and $y$ marginals, respectively. In particular, $\pi^i_*: \Joint\rightarrow \Prob$. 

\begin{prop} \label{Wass est}
Let $m_1,m_2\in \Joint$. Then, the following inequality is true
$$
\max\big\{\dw(\pi^1_*m_1,\pi^1_*m_2),\dw(\pi^2_*m_1,\pi^2_*m_2)\big \}\leq \dw(m_1,m_2).$$
\end{prop}

\begin{proof}
For any $\gamma\in \mathcal{P}_2(\R^d\times\R^d\times\R^d\times\R^d)$ with marginals $m_1$ and $m_2$ we have
  \begin{align*}
\int |(x_1,y_1)-(x_2,y_2)|^2 \gamma(dx_1,dx_2,dy_1,dy_2)&\geq \int |x_1-x_2|^2\gamma(dx_1,dx_2,dy_1,dy_2)\\
&=\int |x_1-x_2|^2\gamma_1(dx_1,dx_2)\geq \dw(\pi^1_*m_1,\pi^1_*m_2),
 \end{align*}
where $\gamma_1$ has marginals $\pi^1_*m_1$ and $\pi^1_*m_2$. Therefore, $\dw(m_1,m_2)\geq \dw(\pi^1_*m_1,\pi^1_*m_2)$. Similarly, we can show that $\dw(m_1,m_2)\geq \dw(\pi^2_*m_1,\pi^2_*m_2)$.\\
\end{proof}

The following propositions describe how the regularity of functions defined on $\Prob$ change when we compose with the projections $\pi^i_*$, $i=1,2$.

\begin{prop} \label{Wass Lip}
Let $U:\Prob\rightarrow \mathbb{R}$ be a Lipschitz function with Lipschitz constant equal to $L$. Then, for $i=1,2$ the function $\tilde{U}_i:\Joint\rightarrow \R$ with $\tilde{U}_i(m)=U(\pi^i_*m)$ is also Lipschitz with the same Lipschtiz constant.
\end{prop}

\begin{proof}
Indeed, let $m_1,m_2\in \Joint$. Then, for $i=1,2$ we have
$$|\tilde{U}(m_1)-\tilde{U}(m_2)|=|U(\pi^i_*m_1)-U(\pi^i_*m_2)|\leq L\dw (\pi^i_*m_1,\pi^i_*m_2)\leq L\dw(m_1,m_2).$$
\end{proof}

\begin{prop}\label{Wass reg}
Let $U:\Prob\rightarrow \mathbb{R}$ be a $C^1$ function. Then, for $i=1,2$ the function $\tilde{U}_i:\Joint\rightarrow \R$ with $\tilde{U}_i(m)=U(\pi^i_*m)$ is also $C^1$ with 
\be \label{regularity 2}
\frac{\delta \tilde{U}_i}{\delta m}(m,x_1,x_2)=\frac{\delta U}{\delta m}(\pi^i_*m,\pi^i(x_1,x_2)),\;\text{ for every }m\in\Joint,\; x_1,x_2\in\R^d.
\ee
\end{prop}

\begin{proof}
We will show the result for $i=1$ as the other case can be shown with an identical argument. For simplicity in the calculations below we drop the index $i$, so $\tilde{U}_i=\tilde{U}$ and $\pi^1=\pi$. Let $m_1,m_2\in\Joint$ and $\pi_*m_1=\mu_1,\; \pi_*m=\mu^2$. We have by the properties of the pushforward
\begin{align*}
\lim_{t\to 0} \frac{ \tilde{U}((1-t)m_1+tm_2)-\tilde{U}(m_1)}{t}&=\lim_{t\to 0} \frac{ U((1-t)\mu_1+t\mu_2)-U(\mu_2)}{t}\\
&=\int_{\R^d} \frac{\delta U}{\delta m}(\mu_1,x_1)(\mu_2-\mu_1)(dx_1)\\
&=\int_{\R^d} \frac{\delta U}{\delta m}(\pi_*m_1,x_1)(\pi_*m_2-\pi_*m_1)(dx_1,dx_2)\\
&=\int_{\R^d} \frac{\delta U}{\delta m}(\pi_*m_1,x_1)(\pi_*m_2-\pi_*m_1)(dx_1)\\
&= \int\int_{\R^d\times\R^d} \frac{\delta U}{\delta m}(\pi_*m_1,\pi(x_1,x_2))(m_2-m_1)(dx_1,dx_2).
\end{align*}
Equation (\ref{regularity 2}) follows. The regularity of $\tilde{U}$ is a direct implication of \eqref{regularity 2} and the fact that $U$ is $C^1$.
\end{proof}

Finally, Proposition \ref{Wass reg} can give us a characterization of a function whose Wasserstein flat derivative is independent of a variable (in analogy to \eqref{regularity 2}). Namely,

\begin{prop} \label{Wass independence}
Let $U:\Joint\rightarrow \R$ be a $C^1$ function such that $\frac{\delta U}{\delta m}(m,x,y)$ is independent of $y$, then there exists a $C^1$ function $\bar{U}: \Prob\rightarrow \R$ such that 
$$U(m)=\bar{U}(\pi^1_*m),\quad \text{for all }m\in\Joint.$$
\end{prop}

\begin{proof}
Let $m_0,m_1\in\Joint$ and set $m_s=(1-s)m_0+sm_1$, where $s\in [0,1]$. We also set $\mu_0=\pi^1_*m_0$ and $\mu_1=\pi^1_*m_1$. We observe that since the flat derivative is independent of $y$
\begin{align*}
U(m_0)-U(m_1)&=\int_0^1\int\int_{\R^d\times\R^d}\frac{\delta U}{\delta m}(m_s,x,y)(m_1-m_0)(dx,dy)ds\\
&= \int_0^1\int_{\R^d}\frac{\delta U}{\delta m}(m_s,x)(\mu_1-\mu_0)(dx)ds,
\end{align*}
therefore if $\mu_0=\mu_1$ we have $U(m_0)=U(m_1)$. This implies that there exists a function $\bar{U}$ defined on $\Prob$ such that $U(m)=\bar{U}(\pi^1_*m)$. It suffices to show that $\bar{U}$ is $C^1$. For any $m_0,m_1\in \Joint$ we have
\begin{align*}
\lim_{t\to 0}\frac{U((1-t)m_0+tm_1)-U(m_0)}{t}&=\int\int_{\R^d\times\R^d}\frac{\delta U}{\delta m}(m_0,x,y)(m_1-m_0)(dx,dy)\\
&= \int_{\R^d}\frac{\delta U}{\delta m}(m_0,x,y)m_1(dx,dy)\\
&= \int_{\R^d}\frac{\delta U}{\delta m}(m_0,x)\mu_1(dx).
\end{align*}
Note that if $m_0'$ is another measure such that $\pi^1_*m_0'=\mu_0$, then from this calculation we deduce that 
$$\int_{\R^d}\frac{\delta U}{\delta m}(m_0,x)\mu_1(dx)=\int_{\R^d}\frac{\delta U}{\delta m}(m_0',x)\mu_1(dx),\quad\text{for any }\mu_1\in\Prob.$$
Consequently, $\frac{\delta U}{\delta m}(m_0,x)$ is independent of the choice of $m_0$ such that $\pi^1_*m_0=\mu_0$. Thus, for $\mu_0,\mu_1\in \Prob$ and any $m_0,m_1\in \Joint$ such that $\pi^1_0m_i=\mu_i,\; i=0,1$, we have
\begin{align*}
\lim_{t\to 0}\frac{\bar{U}((1-t)\mu_0+t\mu_1)-\bar{U}(\mu_0)}{t}&=\lim_{t\to 0}\frac{U((1-t)m_0+tm_1)-U(m_0)}{t}\\
&=\int\int_{\R^d\times\R^d}\frac{\delta U}{\delta m}(m_0,x,y)(m_1-m_0)(dx,dy)\\
&= \int_{\R^d}\frac{\delta U}{\delta m}(m_0,x)(\mu_1-\mu_0)(dx)
\end{align*}
and hence $\frac{\delta\bar{U}}{\delta \mu}(\mu, x)$ exists and is equal to $\frac{\delta U}{\delta m}(m,x)$, where $\mu =\pi^1_*m$.
\end{proof}
\vspace{2mm}

\section{On the McKean-Vlasov control equation}
Let the maps $B,\; \Sigma^1,\; \Sigma^0$ be such that $B\in \mathcal{M}_{b,L}^{\hat{\bm{\zeta}}}(\R^d\times\Prob ;\R^d)$ and $\Sigma^1,\Sigma^0\in \mathcal{M}_{b,L}^{\hat{\bm{\zeta}}}(\R^d\times\Prob ;\R^{d\times})$, where $\hat{\bm{\zeta}}$ is as in introduces at the beginning of section 2.1. In the spaces where \eqref{eq:slow-fastMcKeanVlasov} is posed we consider for $s\in [0,T],\; t\in [s,T]$ the controlled McKean-Vlasov SDE:
\begin{align}\label{controlled McKean}
X^{s,\mu}_t = \eta+\int_s^t B(X^{s,\mu}_{\tau},\L(X^{s,\mu}_{\tau}|\F^0_{\tau}))d\tau+&\int_s^t\Sigma^1(X^{s,\mu}_{\tau},\L(X^{s,\mu}_{\tau}|\F^0_{\tau}))dW_{\tau}\nonumber\\
&+\int_s^t\Sigma^0(X_t^{s,\mu},\L(X^{s,\mu}_{\tau}|\F^0_{\tau}))dB_{\tau}^0,
\end{align}
where $\L(\eta)= \mu$. For $x\in\R^d$ we also consider the process $X_t^{s,x,\mu}$, which is such that
\begin{align}\label{controlled SDE}
X^{s,x,\mu}_t = x+\int_s^t B(X^{s,x,\mu}_{\tau},\L(X^{s,\mu}_{\tau}|\F^0_{\tau}))d\tau+&\int_s^t\Sigma^1(X^{s,x,\mu}_{\tau},\L(X^{s,\mu}_{\tau}|\F^0_{\tau}))dW_{\tau}\nonumber\\
&+\int_s^t\Sigma^0(X_t^{s,x,\mu},\L(X^{s,\mu}_{\tau}|\F^0_{\tau}))dB_{\tau}^0.
\end{align}
Note that $X_t^{s,x,\L(\eta)}\bigg|_{x=\eta}=X_t^{s,\L (\eta)}$. In addition, by standard estimates for conditional McKean-Vlasov SDEs, which are consequences of Gr\"onwall's and Burkholder-Davis-Gundy's inequalities (proved in \cite{carmdel 2}), we have that for any $p\geq 2$ there exists a constant $C_p>0$ such that
\begin{align}
& \mathbb{E}\bigg[\sup_{t\in [s,T]}|X_t^{s,\mu} - X_t^{s,\mu'}|^p\bigg]\leq C_p\dw(\mu,\mu')^p, \label{carm 1}\\
&\mathbb{E}^0\bigg[\sup_{t\in [s,T]} \dw\bigg( \L( X_t^{s,\mu}|\F^0_t),\L(X_t^{s,\mu'} |\F^0_t)   \bigg)^p  \bigg]\leq C_p \dw(\mu,\mu')^p,\label{carm 2}\\
&\mathbb{E}\bigg[\sup_{t\in [s,T]}|X_t^{s,x,\mu}-X_t^{s,x',\mu'}|^p   \bigg]\leq C_p( |x-x'|^p+\dw(\mu,\mu')^p)\label{carm 3}
\end{align}
for every $x,\;x'\in \R^d$ and $\mu,\;\mu'\in \Prob$.\\
\textbf{Notation.} For $\eta\in L^2(\F_s)$ with $\L(\eta)=\mu$, since we have estimate \eqref{carm 3} we may use the notation $X_t^{s,x,\mu}=X_t^{s,x,\L(\eta)}= X_t^{s,x,\eta}$ and $X_t^{s,x,\mu}\bigg|_{x=\eta}=X_t^{s,\mu}=X_t^{s,\L(\eta)}$.
\vspace{2mm}

In this section of the Appendix we will prove the following Theorem:

\begin{thm}\label{main appendix}
Let $B,\; \Sigma^1,\; \Sigma^0$ be as above and $G\in \mathcal{M}_{b,L}^{\bm{\dot{\zeta}}}(\Prob ; \R)$. Then, the function $v:[0,T]\times\Prob\rightarrow \R$ with
$$v(t,\mu)=\mathbb{E}^0\bigg[ G(\L(X_T^{t,\mu}|\F^0_T))\bigg]$$
is in $\mathcal{M}_{b,L}^{\bm{\dot{\zeta}}}([0,T]\times\Prob; \R)$, satisfies $||v||_{\mathcal{M}_b^{\bm{\dot{\zeta}}}}\leq C||G||_{\mathcal{M}_b^{\bm{\dot{\zeta}}}}$ for some constant $C>0$ depending on $T$ and is the unique solution to the PDE
\be \label{MFC app}
\begin{cases}
\dot{v}(t,\mu)+\int_{\R^d}\bigg( B(x,\mu)\cdot D_{\mu}v(t,\mu,x) +\frac{1}{2}(\Sigma^1\Sigma^{1\top}+\Sigma^0\Sigma^{0\top})(s,\mu):D_xD_{\mu}v(t,\mu,x)\bigg) \mu(dx)\\
\hspace{0.7cm}+\frac{1}{2}\int_{\R^{2d}} \Sigma^0(x,\mu)\Sigma^{0\top}(x',\mu):D^2_{\mu\mu}v(t,\mu,x,x')\mu(dx)\mu(dx')=0,\;\;(t,\mu)\in [0,T)\times\Prob,\\
v(T,\mu)=G(\mu),\;\; \mu\in\Prob.
\end{cases}
\ee
Furthermore, the following estimate holds
\be \label{time measure der estimate}
\sup_{\substack{t\in [0,T],\mu\in \Prob \\x_1,x_2\in \R^d}}\bigg \{ |D_m\dot{v}(t,\mu,x_1)|+|D_{x_1}D_m\dot{v}(t,\mu, x_1)|+|D^2_m\dot{v}(t,\mu,x_1,x_2)|\bigg\} \leq C(T)||G||_{\mathcal{M}_b^{\bm{\dot{\zeta}}}}.
\ee
\end{thm}
\vspace{2mm}

The proof is inspired by the ideas presented in \cite{associated SDEs}. For simplicity, we assume that $d=1$ as the argument can be trivially extended for any dimension $d$.

\begin{prop} \label{x differentiability}
Let $X_t^{s,x,\mu}$ be as in \eqref{controlled SDE}. Then, $D_xX_t^{s,x,\mu}$ exists and is the unique solution of the SDE
\begin{align}
D_xX_t^{s,x,\mu}=&\;1+\int_s^t D_xB(X_{\tau}^{s,x,\mu},\mu_{\tau})D_xX_t^{s,x,\mu}d\tau +\int_s^t D_x\Sigma^1(X_{\tau}^{s,x,\mu},\mu_{\tau})D_xX_{\tau}^{s,x,\mu} dW_{\tau}\nonumber\\
&+ \int_s^t D_x\Sigma^0(X_{\tau}^{s,x,\mu},\mu_{\tau})D_xX_{\tau}^{s,x,\mu} dB^0_{\tau},\label{SDE}
\end{align}
where $\mu_t=\L( X_t^{s,\mu}|\F_t^0)$, $t\in [s,T]$. In addition, the derivatives $D^k_xX_t^{s,x,\mu}$ exist for $k=2,3,4$. Furthermore, for every $p\geq 2$, there exists a constant $C_p>0$ such that for all $t\in [s,T]$, $x,x'\in \R$ and $\eta,\eta'\in L^2(\F^1_s)$ with $\L(\eta)=\mu$ and $\L(\eta')=\mu'$
\begin{align}
&\mathbb{E}\bigg[ \sup_{t\in [s,T]} |D^k_xX_t^{s,x,\mu}|^p\bigg]\leq C_p\label{x estimate 1}\\
&\mathbb{E}\bigg[ \sup_{t\in [s,T]}|D^k_xX_t^{s,x,\mu}-D^k_xX_t^{s,x',\mu'}|^p\bigg]\leq C_p( |x-x'|^p+ \dw(\mu,\mu')^p)\label{x estimate 2}
\end{align}
for $k=1,2,3,4$.
\end{prop}

\begin{proof}
Note that \eqref{SDE} is obtained by formally differentiating \eqref{controlled SDE}. Let $D_xX_t^{s,x,\mu}$ be its solution. Then, by a standard Gr\"onwall inequality argument (after invoking the BDG inequality the regularity of $B,\; \Sigma$ and $\Sigma^0$ and \eqref{carm 2}, \eqref{carm 3}) we get that
\begin{align}
\mathbb{E}\bigg[  \sup_{t\in[s,T]}|X_t^{s,x+h,\mu}-X_t^{s,x,\mu}-hD_xX_t^{s,x,\mu}|^2  \bigg]= o(|h|^2),\label{L2 esimate}
\end{align}
therefore
$X_t^{s,x,\mu}$ is differentiable with respect to $x$, in the $L^2$-sense, and its derivative is equal to $D_xX_t^{s,x,\mu}$. Now, for $k=1$, estimates \eqref{x estimate 1} and \eqref{mixed estimate 2} follow by standard estimates for SDEs by \eqref{SDE} due to the regularity of $B,\; \Sigma$ and $\Sigma^0$.\\
For $k=2,3,4$ the existence of $D_x^kX_t^{s,x,\mu}$ follows in a similar way by formally differentiating \eqref{SDE} ($l=1,2,3$ times, respectively) and then using the regularity of $B,\; \Sigma$ and $\Sigma^0$ deriving an estimate like \eqref{L2 esimate}. For $k=2,3,4$, estimates \eqref{x estimate 1} and \eqref{x estimate 2} follow by standard estimates for the SDEs satisfied by $D_x^kX_t^{s,x,\mu}$ and estimates \eqref{x estimate 1} and \eqref{x estimate 2} for the previous derivatives.
\end{proof}

Next we study the smoothness of the map $\eta \mapsto X_t^{s,x,\mu}$, where $\L(\eta)= \mu$. In particular, we show that this map is Fr\'echet differentiable with Fr\'echet derivative given by $\eta\mapsto\bigg(\xi\mapsto \tilde{\mathbb{E}}[U_t^{s,x,\mu}(\tilde{\eta})\tilde{\xi}]\bigg)$, where $U_t^{s,x,\mu}(y)$ is adapted to $(\mathcal{F}_t)_{t\in [s,T]}$ and where $(\tilde{\eta},\tilde{\xi})$ is an independent copy of $(\eta,\xi)$ defined over a new probability space $(\tilde{\Omega},\tilde{\F},\tilde{\mathbb{P}})$. As explained in \cite{CarmonaDelarue} the Lions derivative $D_{\mu}X_t^{s,x,\mu}$ is
$$D_{\mu}X_t^{s,x,\mu}(y):= U_t^{s,x,\mu}(y),\; t\in [s,T],\; x,y\in \R.$$

\begin{prop} \label{measure derivative}
Let $(s,x,\mu)\in [0,T]\times\R\times \mathcal{P}_2(\R)$, $X_t^{s,\mu}$ and $X_t^{s,x,\mu}$ be as in \eqref{controlled McKean}, \eqref{controlled SDE} and $U_t^{s,x,\mu}(y),\; U_t^{s,\mu}(y)$ be the unique solution to the system of (uncoupled) SDEs
\begin{align}
&U_t^{s,x,\mu}(y)=\int_s^t D_xB(X_r^{s,x,\mu},\mu_r)U_r^{s,x,\mu}(y)dr\nonumber\\
&+ \int_s^t\mathbb{\tilde{E}}^1[D_{\mu}B(X_r^{s,x,\mu},\mu_r, \tilde{X}_r^{s,y,\mu})\cdot D_x\tilde{X}_r^{s,y,\mu}+D_{\mu}B(X_r^{s,x,\mu},\mu_r, \tilde{X}_r^{s,\mu})\cdot \tilde{U}_r^{s,\L (\tilde{\eta})}(y)]dr\nonumber\\
&+ \int_s^t\mathbb{\tilde{E}}^1[D_{\mu}\Sigma^1(X_r^{s,x,\mu},\mu_r, \tilde{X}_r^{s,y,\mu})\cdot D_x\tilde{X}_r^{s,y,\mu}+D_{\mu}\Sigma^1(X_r^{s,x,\mu},\mu_r, \tilde{X}_r^{s,\mu})\cdot\tilde{U}_r^{s,\L (\tilde{\eta})}(y)]dW_r\nonumber\\
&+ \int_s^t\mathbb{\tilde{E}}^1[D_{\mu}\Sigma^0(X_r^{s,x,\mu},\mu_r, \tilde{X}_r^{s,y,\mu})\cdot D_x\tilde{X}_r^{s,y,\mu}+D_{\mu}\Sigma^0(X_r^{s,x,\mu},\mu_r, \tilde{X}_r^{s,\mu})\cdot \tilde{U}_r^{s,\L (\tilde{\eta})}(y)]dB^0_r\nonumber\\
&+\int_s^t D_x\Sigma^1(X_r^{s,x,\mu},\mu_r)U_r^{s,x,\mu}(y)dW_r+\int_s^t D_x\Sigma^0(X_r^{s,x,\mu},\mu_r)U_r^{s,x,\mu}(y)dB^0_r,\label{SDE B3}
\end{align}
\begin{align}
&U_t^{s,\mu}(y)=\int_s^t D_xB(X_r^{s,\mu},\mu_r)U_r^{s,\mu}(y)dr\nonumber\\
&+ \int_s^t\mathbb{\tilde{E}}^1[D_{\mu}B(X_r^{s,\mu},\mu_r, \tilde{X}_r^{s,y,\mu})\cdot D_x\tilde{X}_r^{s,y,\mu}+D_{\mu}B(X_r^{s,\mu},\mu_r, \tilde{X}_r^{s,\mu})\cdot \tilde{U}_r^{s,\L (\tilde{\eta})}(y)]dr\nonumber\\
&+ \int_s^t\mathbb{\tilde{E}}^1[D_{\mu}\Sigma^1(X_r^{s,\mu},\mu_r, \tilde{X}_r^{s,y,\mu})\cdot D_x\tilde{X}_r^{s,y,\mu}+D_{\mu}\Sigma^1(X_r^{s,\mu},\mu_r, \tilde{X}_r^{s,\mu})\cdot\tilde{U}_r^{s,\L (\tilde{\eta})}(y)]dW_r\nonumber\\
&+ \int_s^t\mathbb{\tilde{E}}^1[D_{\mu}\Sigma^0(X_r^{s,\mu},\mu_r, \tilde{X}_r^{s,y,\mu})\cdot D_x\tilde{X}_r^{s,y,\mu}+D_{\mu}\Sigma^0(X_r^{s,\mu},\mu_r, \tilde{X}_r^{s,\mu})\cdot \tilde{U}_r^{s,\L (\tilde{\eta})}(y)]dB^0_r\nonumber\\
&+\int_s^t D_x\Sigma^1(X_r^{s,\mu},\mu_r)U_r^{s,\mu}(y)dW_r+\int_s^t D_x\Sigma^0(X_r^{s,x,\mu},\mu_r)U_r^{s,\mu}(y)dB^0_r,\label{SDE B4}
\end{align}
where $\eta\in L^2(\F^1_s)$ with $\L(\eta)=\mu$, $\mu_t=\L(X_t^{s,\mu}|\F^0_t)$, $\tilde{\eta}$ is an independent copy of $\eta$ defined over a new probability space $(\tilde{\Omega}^1,\tilde{\F}^1,\tilde{\mathbb{P}}^1)$ and the process $(\tilde{U}^{s,\L(\tilde{\eta})},W)$ is supposed to follows under $\tilde{\mathbb{P}}^1$ the same law as $(U^{s,\L(\eta)},W)$.
Then, for every $y\in \R$, the Wasserstein derivative $D_{\mu}X_t^{s,x,\mu}(y)$, $y\in \R$, exists and is equal to $U_t^{s,x,\mu}(y)$. Moreover, for all $p\geq 2$, there exists a constant $C_p>0$ such that for all $x,x',y,y'\in\R$ and $\eta,\eta'\in L^2(\F_s)$ with $\L(\eta)=\mu,\; \L(\eta ')=\mu'$ the following estimates hold 
\begin{align}
&\mathbb{E}\bigg[ \sup_{t\in[s,T]}|D_{\mu}X_t^{s,x,\mu}(y)|^p\bigg]\leq C_p,\label{measure der estimate 1}\\
&\mathbb{E}\bigg[ \sup_{t\in[s,T]}( |D_{\mu}X_t^{s,x,\mu}(y)-D_{\mu}X_t^{s,x',\mu'}(y') |^p \bigg]\leq C_p( |x-x'|^p+|y-y'|^p+\dw(\mu,\mu')^p).\label{measure der estimate 2}
\end{align}
\end{prop}

\begin{proof}
In the proof, for simplicity, we assume that $B=0$ and that $\Sigma^1,\Sigma^0$ depend only on the measure variable. The extension to the more general case stated above follows easily by doing the analogous calculations.\\

Let $\eta,\;\xi\in L^2(\F^1_s)$ and $\tilde{\eta},\; \tilde{\xi}$ be copies of $\eta,\;\xi$ defined over $(\tilde{\Omega}^1,\tilde{\F}^1,\tilde{\mathbb{P}}^1)$.

\textit{Step 1:} We first show that for all $(t,x)\in [s,T]\times \R$ the $L^2$- G\^ateaux derivative 
$$L^2-\lim_{h\rightarrow 0}\frac{1}{h}(X_t^{s,x,\L(\eta+h\xi)}-X_t^{s,x,\L(\eta)})$$
exists. In fact, we will show that if $Z_t^{s,x,\L(\eta)}(\xi)$ is the solution of the SDE
\begin{align}
Z_t&^{s,x,\L(\eta)}(\xi)=\int_{s}^t \tilde{\mathbb{E}}^1[ D_{\mu}\Sigma^1( \L(X_r^{s,\L(\eta)}|\F^0_r), \tilde{X}_{r}^{s,\L(\tilde{\eta})})(D_xX_r^{s,y,\L(\tilde{\eta})}\tilde{\xi}+\tilde{Z}_r^{s,y,\L(\tilde{\eta})}(\tilde{\xi}))\bigg|_{y=\tilde{\eta}} ] dW_r\nonumber\\
&+\int_{s}^t \tilde{\mathbb{E}}^1[ D_{\mu}\Sigma^0( \L(X_r^{s,\L(\eta)}|\F^0_r), \tilde{X}_{r}^{s,\L(\tilde{\eta})})(D_xX_r^{s,y,\L(\tilde{\eta})}\tilde{\xi}+\tilde{Z}_r^{s,y,\L(\tilde{\eta})}(\tilde{\xi}))\bigg|_{y=\tilde{\eta}} ] dB^0_r, \label{SDE Z}
\end{align}
then
\be \label{estimate frechet}
\mathbb{E}\bigg[ \sup_{t\in[s,T]}|X_t^{s,x,\L(\eta+h\xi)}-X_t^{s,x,\L(\eta)}-hZ_{t}^{s,x,\L(\eta)}\bigg]\leq Ch^4\mathbb{E}[\xi^2]^2
\ee
for some constant $C>0$. Indeed, due to estimates \eqref{carm 1}, \eqref{carm 2} and \eqref{carm 3}, by the calculation done in Lemma 4.2 of \cite{associated SDEs} we have
\begin{align}
&X_t^{s,x,\L(\eta+h\xi)}-X_t^{s,x,\L(\eta)}-hZ_t^{s,x,\L(\eta)}= R(t,x,h)\label{first}\\
&+\int_s^t\tilde{\mathbb{E}}^1[ D_{\mu}\Sigma^1(\L(X_r^{s,\L(\eta)}|\F^0_r), \tilde{X}_{r}^{s,\L(\tilde{\eta})})(\tilde{X}_r^{s,\L(\tilde{\eta}+h\tilde{\xi})}-\tilde{X}_r^{t,\L(\tilde{\eta})}-h(D_x\tilde{X}_r^{t,z,\L(\eta)}\tilde{\xi}+\tilde{Z}_r^{t,z,\L(\tilde{\eta})}(\tilde{\xi}))\bigg|_{z=\tilde{\eta}}]dW_r\nonumber\\
&+\int_s^t\tilde{\mathbb{E}}^1[ D_{\mu}\Sigma^0(\L(X_r^{s,\L(\eta)}|\F^0_r), \tilde{X}_{r}^{s,\L(\tilde{\eta})})(\tilde{X}_r^{s,\L(\tilde{\eta}+h\tilde{\xi})}-\tilde{X}_r^{t,\L(\tilde{\eta})}-h(D_x\tilde{X}_r^{t,z,\L(\eta)}\tilde{\xi}+\tilde{Z}_r^{t,z,\L(\tilde{\eta})}(\tilde{\xi}))\bigg|_{z=\tilde{\eta}}]dB^0_r,\nonumber
\end{align}
where $R(t,x,h)$ satisfies 
$$\mathbb{E}\bigg[ \sup_{t\in [s,T]} |R(t,x,h)|^2\bigg] \leq Ch^4\mathbb{E}[ |\xi|^2]^2.$$
Using this, to prove \eqref{estimate frechet}, due to the boundness of $D_{\mu}\Sigma^1$ and $D_{\mu}\Sigma^0$ we only need to bound the expected value of the square of the quantity
$$\tilde{X}_r^{s,\L(\tilde{\eta}+h\tilde{\xi})}-\tilde{X}_r^{s,\L(\tilde{\eta})}-h(D_x\tilde{X}_r^{s,z,\L(\eta)}\tilde{\xi}+\tilde{Z}_r^{s,z,\L(\tilde{\eta})}(\tilde{\xi}))\bigg|_{z=\tilde{\eta}}:= I_1(r,h)+I_2(r,h).$$
We have
\begin{align*}
\tilde{\mathbb{E}}[|I_1(r,h)|^2]&=\tilde{\mathbb{E}}[|\tilde{X}_r^{s,z,\L(\tilde{\eta}+h\tilde{\xi})}-\tilde{X}_r^{s,\L(\tilde{\eta})}-h\tilde{Z}_r^{s,z,\L(\tilde{\eta})}(\tilde{\xi}))|^2\bigg|_{z=\tilde{\eta}}]\\
&\leq \sup_{x\in\R}\tilde{\mathbb{E}}[|\tilde{X}_r^{s,x,\L(\tilde{\eta}+h\tilde{\xi})}-\tilde{X}_r^{s,x,\L(\tilde{\eta})}-h\tilde{Z}_r^{s,x,\L(\tilde{\eta})}(\tilde{\xi}))|^2]:=A(r)\quad\quad \text{and}\\
&\hspace{-2cm}\tilde{\mathbb{E}}[|I_2(r,h)|^2]=\tilde{\mathbb{E}}[|X_r^{s,\eta+h\xi,\L(\eta+h\xi)}-X_r^{s,\eta,\L(\eta+h\xi)}-hD_xX_r^{s,\eta,\L(\eta)}|^2]\\
&\leq Ch^4\mathbb{E}[|\eta|^2]^2,
\end{align*}
where in the last inequality we used the fundamental theorem of calculus for the first two terms and \eqref{x estimate 2}. We combine these inequalities with \eqref{first} and we get
$$A(r)\leq Ch^4\mathbb{E}[|\eta|^2]^2 + C\int_s^tA(r)dr.$$
The result now follows from Gr\"onwall's inequality.\\

\textit{Step 2:} Let $\hat{\eta},\hat{\xi}$ be independent copies of $\eta,\;\xi$ and $\tilde{\eta},\;\tilde{\xi}$ defined over a new probability space $(\hat{\Omega},\hat{\F},\hat{\mathbb{P}})$. We claim that
$$Z_t^{s,x,\L(\eta)}(\xi)=\hat{\mathbb{E}}^1[U_t^{s,x,\mu}(\hat{\eta})\hat{\xi}],\;\mathbb{P}-\text{a.s}.$$
The proof is a standard calculation involving Fubini's theorem. Indeed, we can easily see that $Z_t^{s,\L(\eta)}(\xi):=Z_t^{s,x,\L(\eta)}(\xi)\bigg|_{x=\eta}$ satisfies the same SDE as $\hat{\mathbb{E}}^1[U_t^{s,\L(\eta)}(\hat{\eta})\hat{\xi}]$, which is \eqref{SDE Z} when $x=\eta$, therefore, by uniqueness, we must have $Z_t^{s,\L(\eta)}(\xi)=\hat{\mathbb{E}}^1[U_t^{s,\L(\eta)}(\hat{\eta})\hat{\xi}]$ , $\mathbb{P}$-a.s. Using this, we can calculate that $Z_t^{s,x,\L(\eta)}(\xi)$ satisfies SDE as $\hat{\mathbb{E}}^1[U_t^{s,x,\mu}(\hat{\eta})\hat{\xi}]$. The conclusion follows.\\

\textit{Step 3:} Let $p\geq 2$, $x,x',y,y'\in\R$ and $\eta,\eta'\in L^2(\F^1_s)$ with $\L(\eta)=\mu,\; \L(\eta ')=\mu'$. By combining \eqref{x estimate 1}, \eqref{x estimate 2} and \eqref{carm 3} with standard arguments involving Gr\"onwall's and BDG inequalities, we deduce that there exists a constant $C_p>0$ such that
\begin{align}
&\mathbb{E}\bigg[ \sup_{t\in[s,T]}|U_t^{s,x,\mu}(y)|^p\bigg]\leq C_p,\label{measure der estimate 3}\\
&\mathbb{E}\bigg[ \sup_{t\in[s,T]}( |U_t^{s,x,\eta}(y)-U_t^{s,x',\eta'}(y') |^p \bigg]\leq C_p( |x-x'|^p+|y-y'|^p+\dw(\mu,\mu')^p).\label{measure der estimate 4}
\end{align}
We can also easily prove that SDE \eqref{SDE B4} and estimates \eqref{carm 1}, \eqref{carm 3} and \eqref{x estimate 1}, \eqref{x estimate 2} imply
\begin{align}
\mathbb{E}\bigg[ \sup_{t\in [s,T]}|U_t^{s,\mu}(y) -U_t^{s,\mu'}(y')|^p      \bigg]\leq C_p(|y-y'|^p+\dw(\mu, \mu')^p) \label{measure der estimate 5}
\end{align}
By the relation proved in step 2 and  \eqref{measure der estimate 4}, we can easily deduce that the G\^ateaux derivative $Z_t^{s,x,\L(\eta)}$ is continuous in $\eta$, therefore it is the Fr\'echet derivative. Thus, the map $\eta\mapsto X_t^{s,x,\eta}=X_t^{s,x,\L(\eta)}$ is Fr\'echet differentiable, so we can define the Lions derivative of the map $\mu\mapsto X_t^{s,x,\mu}$. By step 2, we have $D_{\mu}X_t^{s,x,\mu}(y)= U_t^{s,x,\mu}(y)$. The proof is complete.
\end{proof}

\begin{rmk}
(i) From estimate \eqref{measure der estimate 4}, we see that $U^{s,x,\eta}$ depends on $\eta$ only through $\L(\eta)$, hence we have the identification $U_t^{s,x,\L(\eta)}=U_t^{s,x,\eta}$ that we were using since the beginning throughout the proof. The same is true for $U_t^{s,\mu}=U_t^{s,\eta}$, due to \eqref{measure der estimate 5}.\\
(ii) Let $y\in \R$. Substituting $\eta$ for $x$ in \eqref{SDE B3}, we deduce from the uniqueness of the solution that $U_t^{s,\L(\eta)}(y)=U_t^{s,x,\L(\eta)}(y)\bigg|_{x=\eta}$, $\mathbb{P}-$a.s. Using the same argument, we can also substitute $\F^1_s$-measurable random variables for $y$ in $U_t^{s,\L(\eta)}(y)$.
\end{rmk}

\begin{cor} \label{key diff}
Let $X_t^{s,\mu}$ be as in \eqref{controlled McKean}. Then, for any $l\in \{0,1,2,3,4\}$ the map $\mu\mapsto D_x^lX_t^{s,x,\mu}$ satisfies the property that for any multi-index $(n,\bm{\beta})$ such that $k=|(n,\bm{\beta})|\leq 4-l$, $D^{(n,\bm{\beta})}D_x^lX_t^{s,x,\mu}$ exists. In addition, if $p\geq 2$ this derivative satisfies
\begin{align}
&\mathbb{E}\bigg[\sup_{t\in [s,T]}|D^{(n,\bm{\beta})}D_x^lX_t^{s,x,\mu}(z_1,...,z_n)|^p\bigg ]\leq C_p, \label{mixed estimate 1}\\
&\mathbb{E}\bigg[ \sup_{t\in [s,T]}|D^{(n,\bm{\beta})}D_x^lX_t^{s,x,\mu}(z_1,...,z_n)-D^{(n,\bm{\beta})}D_x^lX_t^{s,x',\mu'}(z_1',...,z_n')|^p\bigg]\nonumber\\
&\hspace{4cm}\leq C_p(|x-x'|^p+\sum_{i=1}^n|z_i-z_i'|^p+\dw(\mu,\mu')^p)\label{mixed estimate 2}
\end{align}
for  any $x,x',z_1,z_1',...,z_n,z_n'\in \R$ and $\mu,\;\mu'\in \mathcal{P}_2(\R)$ for some constant $C_p>0$.
\end{cor}

\begin{proof}
For $l=0$ and $k=0$ the result follows by relations \eqref{carm 1}, \eqref{carm 3}. For $l=0$ and $k=1$ the result follows by Proposition \ref{measure derivative}. For $l=0$ and $k\in \{2,3,4\}$ the result follows by introducing the adaptation we made in the previous Propositions involving the common noise to the inductive argument made in Theorem 3.4 of \cite{weak quant}. For $l\in \{1,2,3,4\}$ the result follows by the same argument as in Proposition \ref{measure derivative} applied to $D_x^lX_t^{s,x,\mu}$ and the estimates from Proposition \ref{x differentiability}.
\end{proof}

We are now ready to prove Theorem \ref{main appendix}.

\begin{proof}
Let $t\in [0,T]$, $\eta\in L^2(\F^1_t)$ with $\L(\eta)=\mu$ and $\xi\in L^2(\F^1_t)$. \\
\textit{Step 1:} Our first step is to show that for any multi-index $(n,\bm{\beta})$ such that $|(n,\bm{\beta})|\leq 4$ the derivative $D^{(n,\bm{\beta})}v(t,\mu,y_1,...,y_n)$ exists and it is bounded by $C||G||_{\mathcal{M}^{\bm{\dot{\zeta}}}_{b}}$ for some constant $C$ (independent of $(n,\bm{\beta})$ and Lipschitz (with respect to $\mu$ and $y_1,...,y_n$). This will give us $v(t,\cdot)\in \mathcal{M}_{b,L}^{\bm{\dot{\zeta}}}$ and the estimate on $||v||_{\mathcal{M}_{b}^{\bm{\dot{\zeta}}}}$.\\

We observe that for $h\in \R$ and $s\in [t,T]$, due to Propositions \ref{x differentiability} and \ref{measure derivative}, we have
\begin{align*}
D_hX_s^{t,\L (\eta+ h\xi)}&= D_h\bigg ( X_s^{t,x,\L (\eta+ h\xi)}\bigg|_{x=\eta+h\xi}\bigg)\\
&= \bigg( D_x X_s^{t,x,\L(\eta+h\xi)}\bigg|_{x=\eta+h\xi}\bigg) \xi+\bigg( \lim_{r\rightarrow 0}\frac{X_s^{t,x,\L(\eta+(h+r)\xi)}- X_s^{t,x,\L(\eta+h\xi)}}{r}\bigg) \bigg|_{x=\eta+h\xi}.
\end{align*}
Thus, the map $\eta \mapsto X_s^{t,\L(\eta)}$ is Fr\'echet differentiable with directional derivative
$$D_{\eta}(X_s^{t,\L(\eta)})(\xi)=D_hX_s^{t,\L (\eta+ h\xi)}\bigg|_{h=0}= \bigg( D_x X_s^{t,x,\L(\eta)}\bigg|_{x=\eta}\bigg) \xi +D_{\eta}(X_s^{t,x,\L(\eta)})(\xi)\bigg| _{x=\eta}.$$
We define the lifts $G':\L^(\F_T^1)\rightarrow \R$ and $v'(t,\cdot): \L^2(\F^1_t)\rightarrow \R$ of $G,\; v(,\cdot)$, respectively, which are given by $G'(\eta)= G(\L(\eta))$ and $v'(t,\eta)=v(t,\L(\eta))$. Therefore, by the chain rule of Fr\'echet differentiation we have
\begin{align*}
D_{\eta}&v(t,\eta)(\xi)=\mathbb{E}^0\bigg[ DG'(X_T^{t,\L(\eta)})(D_{\eta}X_T^{t,\L(\eta)}(\xi)) \bigg]\\
&=\mathbb{E}^0\bigg[ \mathbb{E}^1\bigg[  D_{\mu}G(\L(X_T^{t,\L(\eta)}|\F_T^0), X_T^{t,\L(\eta)}) \bigg(\bigg( D_x X_s^{t,x,\L(\eta)}\bigg|_{x=\eta}\bigg) \xi +D_{\eta}(X_s^{t,x,\L(\eta)})(\xi)\bigg| _{x=\eta}\bigg)        \bigg]   \bigg]\\
&= \mathbb{E}^0\bigg[ \mathbb{E}^1\bigg[  D_{\mu}G(\L(X_T^{t,\L(\eta)}|\F_T^0), X_T^{t,\L(\eta)}) \bigg(\bigg( D_x X_s^{t,x,\L(\eta)}\bigg|_{x=\eta}\bigg) \xi +\tilde{\mathbb{E}}[D_{\mu}X_T^{t,x,\L(\eta)}(\tilde{\eta})\bigg|_{x=\eta}\tilde{\xi}]\bigg) \bigg] \bigg]\\
&=\tilde{\mathbb{E}}^1\bigg[\mathbb{E}^0\bigg[\mathbb{E}^1\bigg[  D_{\mu}G(\L(X_T^{t,\L(\eta)}|\F_T^0), X_T^{t,x,\L(\eta)}) D_x X_s^{t,x,\L(\eta)}\bigg|_{x=\tilde{\eta}} \tilde{\xi}\bigg] \bigg]\bigg]\\
&\quad\quad\quad+ \tilde{\mathbb{E}}^1\bigg[\mathbb{E}^0\bigg[\mathbb{E}^1\bigg[
D_{\mu}G(\L(X_T^{t,\L(\eta)}|\F_T^0), X_T^{t,\L(\eta)})D_{\mu}X_T^{t,x,\L(\eta)}(\tilde{\eta})\bigg|_{x=\eta}
\bigg] \bigg]\tilde{\xi}\bigg],
\end{align*}
where in the last equality we used that $X^{t,x,\L(\eta)}_T$ is independent of $\F^1_t$, $\eta,\; \xi$ are $\F^1_t$-measurable and that $(\tilde{\eta}, \tilde{\xi})$ is independent of $(\eta,\xi)$ and of the same law under $\tilde{\mathbb{P}}$ as $(\eta,,\xi)$ under $\mathbb{P}$. This equality implies
\be \label{first derivative for final}
D_{\mu}v(t,\mu,y)=\mathbb{E}^0\bigg[\mathbb{E}^1[ D_{\mu}G(\L(X_T^{t,\mu}|\F^0_T),X_T^{t,y,\mu})D_xX_T^{t,y,\mu}+D_{\mu}G(\L(X_T^{t,\mu}|\F^0_T),X_T^{t,\mu})U_T^{t,\mu}(y)]\bigg],
\ee
where $U_T^{t,\mu}$ is defined in Proposition \ref{measure derivative}. Since $G\in \mathcal{M}_b^{\bm{\dot{\zeta}}}$ and \eqref{x estimate 1}, \eqref{x estimate 2}, \eqref{measure der estimate 1} and \eqref{measure der estimate 5} hold, we see that $D_{\mu}v(t,\mu,y)$ is bounded by $C||G||_{\mathcal{M}^{\bm{\dot{\zeta}}}_{b,L}}$ (for some constant $C$) and is also Lipschitz with respect to $\mu$ and $y$. \\
To verify the differentiability of higher order we may use Corollary \ref{key diff}. Indeed, to compute the higher order derivatives with respect to the measure variable, we lift $D_{\mu}G(\cdot,y),\;D_{\mu}v(t,\cdot,y)$ to functions $(D_{\mu}G)'(\cdot,y),\;(D_{\mu}v)'(t,\cdot,y)$ defined over $L^2(\F^1_t)$ so that 
\begin{multline*}
(D_{\mu}v)'(t,\eta,y)=\mathbb{E}^0\bigg[\mathbb{E}^1[ (D_{\mu}G)'(X_T^{t,\L(\eta)},X_T^{t,y,\L(\eta)})D_xX_T^{t,y,\L(\eta)}+\\(D_{\mu}G)'(X_T^{t,\L(\eta)},X_T^{t,\L(\eta)})D_{\mu}X_T^{t,x,\L(\eta)}(y)\bigg|_{x=\eta}]\bigg],
\end{multline*}
and, as previously, we compute the Fr\'echet derivative, which will then give us $D^2_{\mu\mu}v$. To compute $D^3_{\mu\mu\mu}v$ and $D^4_{\mu\mu\mu\mu}v$, we follow the same procedure. The computation of the derivatives with respect to the rest of the variables is straightforward by Corollary \ref{key diff}. The fact that the higher order derivatives will be bounded and Lipschitz follows from \eqref{mixed estimate 1} and \eqref{mixed estimate 2}.\\


\noindent
\textit{Step 2:} Our second step is to prove that $v(t,\mu)$ is differentiable with respect to $t$ and that the derivative is bounded. Indeed, by the uniqueness of the conditional distribution we have
\begin{align*}
v(t,\mu)=\mathbb{E}^0\bigg[ G(\L(X_{T-t}^{\mu}|\F^0_{T-t})) \bigg].
\end{align*}
The desired differentiability now follows from It\^o's formula (Theorem 4.14 of \cite{carmdel 2}), which provides us with a formula for $\dot{v}$ and, furthermore,
$$\sup_{t\in [0,T],\;\mu\in \mathcal{P}_2(\R)}|\dot{v}(t,\mu)|\leq C||G||_{\mathcal{M}_b^{\bm{\dot{\zeta}}}}$$
follows from the boundness of $B,\; \Sigma^1,\; \Sigma^0$ and the boundness of $G$ and its derivatives.\\

\textit{Step 3:} We now show that $v$ is the unique classical solution to \eqref{MFC app}. Indeed, let $\mu\in \mathcal{P}_2(\R)$, $t\in [0,T]$, $V$ satisfying \eqref{MFC app} in the classical sense and $\mu_{s}=\L(X_{s}^{t,\mu}|\F^0_s)$. Then, by It\^o's formula we have
\begin{align*}
\mathbb{E}^0\bigg[V(T,\mu_T)&-V(t,\mu_t)\bigg]=\int_t^T\mathbb{E}^0\bigg[ -\dot{V}(s,\mu_s) -\int_{\R}B(x,\mu_s)\cdot D_\mu V(s,\mu_s,x)\mu_s(dx)\bigg]ds\\
&+ \int_t^T\mathbb{E}^0\bigg[\frac{1}{2}(\Sigma^1\Sigma^{1\top}(x,\mu)+\Sigma^0\Sigma^{0\top} (x,\mu_s)):D_x D_\mu V(s,\mu_s,x))\mu_s(dx)\bigg] ds\nonumber\\ 
& -\int_t^T\mathbb{E}^0\bigg[\frac{1}{2}\int_{\R^{2}}\Sigma^0(x,\mu_s)\Sigma^{0\top} (x',\mu_s): D^2_{\mu\mu}v(t,\mu_s,x,x')\mu_s(dx)\mu_s(dx')\bigg] ds.
\end{align*}
Since $V$ satisfies \eqref{MFC app}, the right hand side is $0$, so
$$V(t,\mu)= \mathbb{E}^0[G(\mu_T)]=\mathbb{E}^0[G(\L(X_T^{t,\mu}|\F^0_T))]=v(t,\mu).$$

\textit{Step 4:} We finally show estimate \eqref{time measure der estimate}.
Since $v\in \mathcal{M}_{b,L}^{\bm{\dot{\zeta}}}$ and $v$ satisfies \eqref{MFC app} (step 3) , we have that the derivatives $D_{\mu}\dot{v}(t,\mu,x), D_xD_{\mu}\dot{v}(t,\mu,x)$ and $D^2_{\mu\mu}\dot{v}(t,\mu,x_1,x_2)$ exist and can be computed by differentiating 
\begin{align*}
&-\int_{\R^d}\bigg( B(x,\mu)\cdot D_{\mu}v(t,\mu,x) +\frac{1}{2}(\Sigma^1\Sigma^{1\top}+\Sigma^0\Sigma^{0\top})(s,\mu):D_xD_{\mu}v(t,\mu,x)\bigg) \mu(dx)\\
&\hspace{1.5cm}-\frac{1}{2}\int_{\R^{2d}} \Sigma^0(x,\mu)\Sigma^{0\top}(x',\mu):D^2_{\mu\mu}v(t,\mu,x,x')\mu(dx)\mu(dx').
\end{align*}
Estimate \eqref{time measure der estimate} now follows from the boundness of $B,\; \Sigma^1,\; \Sigma^0$ and their derivatives and the boundness of $||v||_{\mathcal{M}_b^{\bm{\dot{\zeta}}}}$ derived in Step 1.

\section{On the choice of the Assumptions}
In this section we will provide some comments related to Remark \ref{remark for appendix}(vi). \\

We assume that in \eqref{eq:slow-fastMcKeanVlasov} instead of $\frac{\lambda}{\ep}f(X^{\ep}_t,Y^{\ep}_t,\L(X_t^{\ep}|\F_t^0))$ we have $\frac{1}{\ep}b(X^{\ep}_t,Y^{\ep}_t,\L(X_t^{\ep}|\F_t^0))$ for some function $b\in \mathcal{M}_p^{\hat{\bm{\zeta}}}$ such that $\bar{b}(x,\mu)=0$ (as in \cite{Spil}). For simplicity, we also assume that $c=g=0$. We consider the function $\Phi\in \mathcal{M}_p^{\hat{\bm{\zeta}}}(\R^d\times\R^d\times\Prob; \R^d)$ which is such that $L_0^{x,\mu}\Phi(x,y,\mu)=-b(x,y,\mu)$. Note that such $\Phi$ exists by Proposition \ref{poisson for calculations}.\\

Keeping the same notation as in section 1.2, assuming that the convergence holds, we want to find a function $u$ such that $u^{\ep}(t,m)$ converges to $u(t,\pi^1_*m)$ as $\ep\rightarrow 0$. By repeating the construction of the correctors $u_1,\; u_2$, which, for reference, is sketched in section 5.1, we find that $u_1(t,m)=\int_{\R^{2d}}\Phi(x,y,\mu)\cdot D_mu(t,\pi^1_*m,x)m(dx,dy)$ and, to construct $u_2$, that $u$ will satisfy 
\be \label{app1}
\begin{cases}
\dot{u}(t,\mu)+\L u(t,\mu)=0,& (t,\mu)\in [0,T)\times \Prob,\\
u(T,\mu)=G(\mu),& \mu\in \Prob
\end{cases}
\ee
with 
\begin{align*}
&\L u(t,\mu) = \int_{\R^d} \bigg( \overline{(D_x\Phi b)}+ \overline{(\tau_1\sigma^{\top}:D^2_{xy}\Phi)}\bigg) (x,\mu) \cdot D_{\mu}u(t,\mu,x)\mu(dx)\\
&+ \int_{\R^d}\overline{(\sigma_1^0\sigma_2^{0\top}:D^2_{xy}\Phi)}(x,\mu)\cdot D_{\mu}u(t,\mu,x)\mu(dx)\\
&+ \int_{\R^{2d}} \int_{\R^d} \overline{ D_{\mu}\Phi(x,y,\mu,x')b(x',y',\mu)} \mu(dx')\cdot D_{\mu}u(t,\mu,x)\mu(dx)\\
&+ \int_{\R^d}\int_{\R^d}\overline{\sigma_2^0(x,\mu)\sigma_1^{0\top}(x',\mu):D_yD_{\mu}\Phi(x,y,\mu,x')}\mu(dx')\cdot D_{\mu}u(t,\mu,x)\mu(dx)\\
&+\frac{1}{2}\int_{\R^{2d}}\bigg( \overline{\sigma\sigma^{\top}}+\frac{\overline{b\otimes\Phi}+\overline{(b\otimes\Phi)^{\top}}}{2}\bigg) (x,\mu) :D_xD_{\mu}u(t,\mu,x) \mu(dx)\\
&+\frac{1}{2}\int_{\R^d}  \frac{\overline{D_y\Phi\tau_1\sigma^{\top}}+\overline{(D_y\Phi\tau_1\sigma^{\top})^{\top}}}{2}(x,\mu):D_xD_{\mu}u(t,\mu,x) \mu(dx)\\
&+\frac{1}{2}\int_{\R^d}\bigg( \sigma_1^0\sigma_1^{0\top}+2\sigma_1(\overline{D_y\Phi}\sigma_2^0)^{\top}\bigg)(x,\mu):D_xD_{\mu}u(t,\mu,x)\mu(dx)\\
&+\frac{1}{2}\int_{\R^{2d}}( \sigma_1^0(x,\mu)\sigma_1^{0\top}(x',\mu)+\overline{(D_y\Phi\sigma_2^0)(x,y,\mu)(D_y\Phi\sigma_2^0)^{\top}(x',y',\mu)}):D^2_{\mu\mu}u\; \mu(dx)\mu(dx')\\
&+\int_{\R^{2d}}(\sigma_1^0(x,\mu)\overline{(D_y\Phi\sigma_2^0)^{\top}(x',y',\mu)}):D^2_{\mu\mu}u(t,\mu,x,x')\mu(dx)\mu(dx'),
\end{align*}
where, as earlier, $\overline{h}(x,x')=\overline{h(x,x',y,y')}=\int_{\R^{2d}}h(x,x',y,y')\Pi_{x,x',\mu}(dy,dy')$ with $\Pi_{x,x',\mu}$ being the invariant measure defined after \eqref{L02}. We observe that $u$ satisfies an equation of the form
\begin{multline}\label{equation}
\dot{u}(t,\mu)+\int_{\R^d}(B(x,\mu)\cdot D_{\mu}u(t,\mu,x)+\Sigma_1(x,\mu):D_xD_{\mu}u(t,\mu,x))\mu(dx)\\+\frac{1}{2}\int_{\R^{2d}}\Sigma_0(x,x',\mu):D^2_{\mu\mu}u(t,\mu,x,x')\mu(dx)\mu(dx')=0
\end{multline}
with $B,\;\Sigma_1,\;\Sigma_0$ being regular. To our knowledge, it is not know if, in general, there exists solution $u$ to such a problem and if the solution will be regular as well, unless we can identify \eqref{equation} with a McKean-Vlasov control equation associated with a conditional McKean-Vlasov process. To do that, $\Sigma_0$ has to satisfy $\Sigma_0(x,x',\mu)=\Sigma_2(x,\mu)\Sigma_2^{\top}(x',\mu)$ for some other matrix valued function $\Sigma_2$, for every $x,x'\in \R^d$ and $\mu\in \Prob$. In our case
\begin{align*}
\Sigma_0(x,x',\mu)&=\sigma_1^0(x,\mu)\sigma_1^{0\top}(x',\mu)+\overline{(D_y\Phi\sigma_2^0)(x,y,\mu)(D_y\Phi\sigma_2^0)^{\top}(x',y',\mu)}\\
&+\sigma_1^0(x,\mu)\overline{(D_y\Phi\sigma_2^0)^{\top}}(x',\mu)+\overline{(D_y\Phi\sigma_2^0)}(x,\mu)\sigma_1^{0\top}(x',\mu).
\end{align*}
Even though this is close to being equal to 
$$(\sigma^0_1(x,\mu)+ \overline{D_y\Phi}\sigma_2^0(x,\mu))(\sigma^0_1(x',\mu)+ \overline{D_y\Phi}\sigma_2^0(x',\mu))^{\top},$$
this is not true because 
\be \label{non eq}
\overline{(D_y\Phi\sigma_2^0)(x,y,\mu)(D_y\Phi\sigma_2^0)^{\top}(x',y',\mu)}\neq \overline{D_y\Phi}\sigma_2^0(x,\mu)(\overline{D_y\Phi}\sigma_2^0)^{\top}(x',\mu),
\ee
since $\Pi_{x,x',\mu}$ is not the product measure $\pi_{x,\mu}\otimes \pi_{x',\mu}$. However, if $D_y\Phi$ is independent of $y$ (which is the case in our result), \eqref{non eq} holds as an equality, as we are averaging over a function which is independent of $y,y'$. This justifies our choice of coefficients in \eqref{eq:slow-fastMcKeanVlasov}.\\

\textbf{Acknowledgments.}\\
The author was partially supported by Panagiots E. Souganidis'  National Science Foundation DMS-1266383 and DMS-1600129, Office for Naval Research N000141712095 and Air Force Office for Scientific Research FA9550-18-1-0494 grants. The author wishes to thank Panagiotis E. Souganidis, Ben Seeger and Nikiforos Mimikos-Stamatopoulos for valuable discussions and the Institute of Mathematical and Statistical Innovation (IMSI) for its hospitality during the long program in the spring of 2023.

\bibliographystyle{siam}

\begin{thebibliography}{10}
\bibitem{CarmonaDelarue} Carmona, Ren\'e, and Fran\c{c}ois Delarue. Probabilistic theory of mean field games with applications I. Switzerland: Springer Nature, 2018..

\bibitem{carmdel 2} Carmona, Ren\'e, and Fran\c{c}ois Delarue. Probabilistic theory of mean field games with applications II. Switzerland: Springer Nature, 2018.

\bibitem{Spil} Bezemek, Zachary, and Konstantinos Spiliopoulos. "Rate of homogenization for fully-coupled McKean-Vlasov SDEs." arXiv preprint arXiv:2202.07753 (2022).

\bibitem{Pardoux} Pardoux, \'Etienne, and Yu Veretennikov. "On the Poisson equation and diffusion approximation. I." The Annals of Probability 29.3 (2001): 1061-1085.

\bibitem{Pardoux2} Pardoux, E., and A. Yu Veretennikov. "On Poisson equation and diffusion approximation 2." The Annals of Probability 31.3 (2003): 1166-1192.

\bibitem{crisan} Crisan, Dan, and Eamon McMurray. "Smoothing properties of McKean-Vlasov SDEs." Probability Theory and Related Fields 171 (2018): 97-148.

\bibitem{Hong1} Hong, Wei, Shihu Li, and Xiaobin Sun. "Diffusion Approximation for Multi-Scale McKean-Vlasov SDEs Through Different Methods." arXiv preprint arXiv:2206.01928 (2022).

\bibitem{Li} Li, Yun, Fuke Wu, and Longjie Xie. "Poisson equation on Wasserstein space and diffusion approximations for McKean-Vlasov equation." arXiv preprint arXiv:2203.12796 (2022).

\bibitem{Rockner1} R\"ockner, Michael, Xiaobin Sun, and Yingchao Xie. "Strong convergence order for slow-fast McKean-Vlasov stochastic differential equations." (2021): 547-576.

\bibitem{Lions} Lions, Pierre-Louis. "Cours au college de france." Available at www. college-de-france. fr (2007).

\bibitem{Rockner2} R\"ockner, Michael, and Longjie Xie. "Diffusion approximation for fully coupled stochastic differential equations." (2021): 1205-1236.

\bibitem{Qiao} Qiao, Huijie, and Wanlin Wei. "Efficient filtering for multiscale McKean-Vlasov Stochastic differential equations." arXiv preprint arXiv:2206.05037 (2022).

\bibitem{Pham} Pham, Huy\^en, and Xiaoli Wei. "Dynamic programming for optimal control of stochastic McKean--Vlasov dynamics." SIAM Journal on Control and Optimization 55.2 (2017): 1069-1101.

\bibitem{ran1} Wang, Wei, A. J. Roberts, and Jinqiao Duan. "Large deviations and approximations for slow-fast stochastic reaction-diffusion equations." Journal of Differential Equations 253.12 (2012): 3501-3522.

\bibitem{ran2} Mastny, Ethan A., Eric L. Haseltine, and James B. Rawlings. "Two classes of quasi-steady-state model reductions for stochastic kinetics." The Journal of chemical physics 127.9 (2007): 094106.

\bibitem{ran3} Weinan, E., Di Liu, and Eric Vanden-Eijnden. "Analysis of multiscale methods for stochastic differential equations." Communications on Pure and Applied Mathematics 58.11 (2005): 1544-1585.

\bibitem{ran4} Harvey, E., Kirk, V., Wechselberger, M., and Sneyd, J. (2011). Multiple timescales, mixed mode oscillations and canards in models of intracellular calcium dynamics. Journal of nonlinear science, 21, 639-683.

\bibitem{Cardaliaguet master eq} Cardaliaguet, P., Delarue, F., Lasry, J. M., \& Lions, P. L. (2019). The master equation and the convergence problem in mean field games:(ams-201). Princeton University Press.

\bibitem{associated SDEs} Buckdahn, R., Li, J., Peng, S., \& Rainer, C. (2017). Mean-field stochastic differential equations and associated PDEs.

\bibitem{weak quant} Chassagneux, Jean-Fran\c{c}ois, Lukasz Szpruch, and Alvin Tse. "Weak quantitative propagation of chaos via differential calculus on the space of measures." The Annals of Applied Probability 32.3 (2022): 1929-1969.

\bibitem{Chen} Chen, Peng, Jianya Lu, and Lihu Xu. ``Approximation to stochastic variance reduced gradient Langevin dynamics by stochastic delay differential equations.'' Applied Mathematics \& Optimization 85.2 (2022): 15.

\end{thebibliography}

\end{document}